\documentclass[11pt, reqno]{amsart}
\usepackage[margin=1.2in]{geometry}
\usepackage{amsmath}
\usepackage[active]{srcltx}
\usepackage{amssymb}
\usepackage[dvips]{graphicx}

  \newtheorem{Corollary}{Corollary}[section]
  \newtheorem{Lemma}{Lemma}[section]
   \newtheorem{Theorem}{Theorem}[section]

\usepackage{enumerate}

\usepackage{lipsum}
\usepackage{marvosym}
\usepackage{zref-savepos}
\usepackage{graphics}

\begin{document}

\title{Substitutions, Tiling Dynamical Systems and Minimal Self-Joinings} 
\author[Y. Son]{YOUNGHWAN SON \\
 Faculty of Mathematics and Computer Science  \\
 Weizmann Institute of Science, Rehovot 7610001, Israel\\
\lowercase{email: younghwan.son@weizmann.ac.il}}
\bigskip

\gdef\shorttitle{ }
\maketitle

\setcounter{section}{0}


\begin{abstract} 
We investigate substitution subshifts and tiling dynamical systems arising from the substitutions (1) $\theta: 0 \rightarrow 001, 1 \rightarrow 11001$ and (2) $\eta: 0 \rightarrow 001, 1 \rightarrow 11100$.  We show that the substitution subshifts arising from $\theta$ and $\eta$ have minimal self-joinings and are mildly mixing. We also give a criterion for 1-dimensional tiling systems arising from $\theta$ or $\eta$ to have minimal self-joinings. We apply this to obtain examples of mildly mixing 1-dimensional tiling systems.
 \end{abstract}

\section{Introduction}

A measure preserving system, denoted by $\bold{X} = (X, \mathcal{X}, \mu, T_G)$, consists of a probability space $(X, \mathcal{X}, \mu)$ together with a measure preserving action $T_G = (T_g)_{g \in G}$ of a locally compact group $G$. In this paper we mainly consider the cases $G = \mathbb{Z}$ and $G = \mathbb{R}$, and denote the corresponding measure preserving system by $\bold{X} = (X, \mathcal{X}, \mu, T)$ or $\bold{X} = (X, \mathcal{X}, \mu, (T_t)_{(t \in \mathbb{R})})$ respectively. For $G= \mathbb{R}$, we also assume that $G$ acts continuously in the sense that for every $A \in \mathcal{X}$ and $\epsilon > 0$, there exists a neighborhood $V$ of the identity in $G$ such that $\mu(A \triangle T_g A) < \epsilon$ for $g \in V$. 
A measure preserving system $\bold{X}$ is called 
\begin{enumerate}
\item {\it{ergodic}} if any set $A$ with $\mu(A \triangle T_gA)=0$ for all $g\in G$ is of measure $0$ or $1$.
\item {\it{weakly mixing}} if the diagonal action $(T_g \times T_g)_{g \in G}$ on $(X \times X, \mathcal{X} \otimes \mathcal{X}, \mu \otimes \mu)$ is ergodic. For an abelian group $G$, it is equivalent to the fact that  the only eigenfunctions of $T_G$ are constant functions. Recall that $f \in L^2(X)$ is called an eigenfunction if $f(T_g x) = \lambda(g) f(x)$ for some $\lambda \in \hat{G}$, the dual group of $G$.  
\item {\it{mildly mixing}} if there is no non-constant rigid function. Recall that a function $f \in L^2(X)$ is called {\it{rigid}} if there is a sequence $t_i \in G$ with $t_i \rightarrow \infty$ such that $f \circ T_{t_i}$ converges to $f$ in $L^2(X)$. (We write $t_i \rightarrow \infty$ if, for any compact set $K \subset G$, there is a positive integer $N$ such that $t_i \notin K$ if $i \geq N$.) 
\item {\it{rigid}} if there is a sequence $t_i \in G$ with $t_i \rightarrow \infty$ such that for every $f \in L^2(X)$, $f \circ T_{t_i}$ converges to $f$ in $L^2(X)$.
\item {\it{strongly mixing}} if  for any $f_1, f_2 \in L^2(X)$ and for any sequence $(t_i)$ with $t_i \rightarrow \infty,$  $$\lim\limits_{t_i \rightarrow \infty} \int f_1(x) \, f_2(T_{t_i} x) \, d \mu = \int f_1 \, d \mu \int f_2 \, d\mu.$$
\end{enumerate}
Note that strong mixing implies mild mixing and mild mixing implies weak mixing.
 
In this paper we will consider the substitutions (1) $\theta: 0 \rightarrow 001, 1 \rightarrow 11001$ and (2) $\eta: 0 \rightarrow 001, 1 \rightarrow 11100$. Specifically we will investigate the mixing properties of the substitution subshifts and tilings systems arising from these substitutions. To facilitate the discussion, we present now some basic facts on substitution dynamical systems and tiling dynamical systems.
 
  An {\it{alphabet}} $\mathcal{A}$ is a finite set  of symbols called {\it{letters}}, denoted as $\mathcal{A} = \{ 0,1,2, \dots, s-1 \}$ or as letters $\mathcal{A} =\{ a, b, \dots \}$. A finite string of letters is called a {\it{word}} or {\it{block}} and the set of all finite words over $\mathcal{A}$ is denoted by $\mathcal{A}^* = \bigcup_{k \geq 1} \mathcal{A}^k$. Elements of $\mathcal{A}^{\mathbb{Z}}$ are called {\it{sequences}} over the alphabet $\mathcal{A}$. 
 One can define a metric $d$ on $\mathcal{A}^{\mathbb{Z}}$ by $d(x,y) = 2^{-n}$, where $n$ is the smallest integer with $x_n \ne y_n$ or $x_{-n} \ne y_{-n}$ for $x=(x_n)$ and $y=(y_n)$ in $\mathcal{A}^{\mathbb{Z}}$. With this metric $\mathcal{A}^{\mathbb{Z}}$ is a compact metric space and the shift map $T$ given by $(Tx)_n = x_{n+1}$ for a sequence $x=(x_n)$ is a homeomorphism on $\mathcal{A}^{\mathbb{Z}}$. A pair $(X,T)$, where $X$ is a closed $T$ invariant subset of $\mathcal{A}^{\mathbb{Z}}$, is called a {\it{subshift}}.
 
  A {\it{substitution}} $\zeta$ is a map from $\mathcal{A}$ to $\mathcal{A}^*$, which induces a map from $\mathcal{A}^*$ to $\mathcal{A}^*$ given by $\zeta (b_0 b_1 \cdots b_n) = \zeta(b_0) \zeta(b_1) \cdots \zeta(b_n).$ 
   A word $u$ is $\zeta$-{\it{admissible}} if $u$ occurs in $\zeta^m(a)$ for some $m \in \mathbb{N}$ and $a \in \mathcal{A}$. Define the substitution space $X_{\zeta}$ as the set of all sequences $x=(x_n) \in \mathcal{A}^{\mathbb{Z}}$ such that every word $x_i x_{i+1} \cdots x_j$ $(i \leq j)$ is $\zeta$-admissible.  A substitution $\zeta$ is called {\it{primitive}} if there exists $m \in \mathbb{N}$ such that for any $i,j \in \mathcal{A}$, $i$ occurs in $\zeta^m(j)$. It is well-known that if a substitution $\zeta$ is primitive, then $(X_{\zeta}, T)$ is minimal and has a unique $T$-invariant probability measure $\mu$ on the Borel $\sigma$-algebra $\mathcal{X}$ of $X$ (see \cite{Q}). A measure preserving system of the form $(X_{\zeta}, \mathcal{X}, \mu,T)$ is called a {\it{substitution subshift}} or a {\it{substitution dynamical system}}.

 One dimensional substitution tiling systems were investigated by Berend and Radin (\cite{BR}) and Clark and Sadun (\cite{CS}). A {\it{tile}} $I =[a,b]$ in $\mathbb{R}$ is a closed interval with the positive length $|I| = b-a >0$. 
A {\it{tiling}} of $\mathbb{R}$ is a set $\mathcal{T}$ of tiles such that $\mathbb{R} = \bigcup \{I : I \in \mathcal{T}\}$ and distinct tiles have non-intersecting interiors. We also assume that the tiling has finitely many tiles up to translations.
  A finite collection of tiles in $\mathcal{T}$ is called a {\it{$\mathcal{T}$-patch}} and  two patches $P_1$ and $P_2$ are said to be {\it{equivalent}} if $P_2 = P_1 + t$ for some $t \in \mathbb{R}$. For a tiling $\mathcal{T}$ of $\mathbb{R}$, {\it{the tiling space}} $X_{\mathcal{T}}$ is defined as the set of all tilings $\mathcal{S}$ of $\mathbb{R}$ with the property that every $\mathcal{S}$-patch is equivalent to some $\mathcal{T}$-patch.
 Define a metric on $X_{\mathcal{T}}$ as following:
$$d (\mathcal{T}_1, \mathcal{T}_2) = \min \{ \frac{1}{\sqrt{2}}, \overline{d}(\mathcal{T}_1, \mathcal{T}_2) \},$$
where $$\overline{d}(\mathcal{T}_1, \mathcal{T}_2) = \inf \{ \epsilon : \, \textrm{for some} \, t \in \mathbb{R}^d  \, \textrm{with} \, |t| < \epsilon, \mathcal{T}_1 + t \,\, \textrm{and} \, \mathcal{T}_2 \, \textrm{agree on} \, B (0, 1/\epsilon )    \}.$$
Then $X_{\mathcal{T}}$ is a compact space and the translation action $T_t (\mathcal{S}) = \mathcal{S} - t$ is continuous for each $t \in \mathbb{R}$.

Given a sequence $x =(x_n) \in \{0,1, \dots, s-1\}^{\mathbb{Z}}$ and a collection of intervals $\mathcal{I} = \{J_0, J_1, \dots , J_{s-1} \}$, a {\it{tiling}} $\mathcal{T}_x := \{I_i : i \in \mathbb{Z}\}$ of $\mathbb{R}$ is obtained by 
\begin{enumerate}
\item taking $I_i$ to be a translate of $J_{x_i}$,
\item putting the left end point of $I_0$ at the origin of $\mathbb{R}$,
\item requiring that $\mathbb{R} = \cup I_i$ and $I_i \cap I_{i+1}$ is a singleton for each $i \in \mathbb{Z}$.
\end{enumerate}  

A tiling $\mathcal{T}$ is called {\it{substitution tiling}} if it is obtained from a sequence $x $ in a substitution space $X_{\zeta}$. If a substitution $\zeta$ is primitive, then there exists a unique probability measure $\mu$ on the Borel $\sigma$-algebra $\mathcal{X}$ of the tiling space $X_{\mathcal{T}}$. A measure preserving system of the form $(X_{\mathcal{T}}, \mathcal{X}, \mu, (T_t)_{t \in \mathbb{R}})$ arising from a substitution tiling $\mathcal{T}$ is called a {\it{substitution tiling system}}. Note that the definition of a substitution tiling space does not depend on the choice of the sequence $x$ in the substitution space.

Dekking and Keane considered in \cite{DK} substitutions (1) $\theta: 0 \rightarrow 001, 1 \rightarrow 11001$ and (2) $\eta: 0 \rightarrow 001, 1 \rightarrow 11100$ and showed that the substitution dynamical systems arising from $\theta$ and $\eta$ are weakly mixing but not strongly mixing. The substitution $\theta$ is obtained by writing $a=00$ and $b=1$ in the Toeplitz substitution $a \rightarrow abab, b \rightarrow bbab$ (\cite{JaKe}, \cite{Kak}) and it had been shown by Kakutani \cite{Kak} that the symbolic dynamical system obtained by doubling the symbol $a$ in the Toeplitz substitution is weakly mixing.  
Moreover, by using the method in \cite{BR} and \cite{CS}, one can show that the tiling systems arising from these substitutions and two intervals $J_0$ and $J_1$, where $\frac{|J_0|}{|J_1|}$ is irrational, are weakly mixing (cf. Lemma in \cite{BR} and Theorem 2.5 in \cite{CS}) but not strongly mixing (cf. Theorem 2.2 in \cite{CS}).

These results lead to the natural question of whether these systems are mildly mixing. In this paper we will answer this question in the affirmative by showing that these systems have minimal self-joinings, which implies the mild mixing property.

 The notion of {\it{minimal self-joinings}} (MSJ) was introduced by Rudolph \cite{Ru1}. 
 A two-fold {\it{self-joining}} of an ergodic system $\bold{X} = (X, \mathcal{X}, \mu, T_G)$  is a probability measure $\lambda$ on $X \times X$ such that $\lambda$ is invariant under the diagonal action $T_G \otimes T_G$ and the marginals of $\lambda$ on each $X$ are $\mu$. An ergodic system $\bold{X} = (X, \mathcal{X}, \mu, T_G)$ is said to have minimal self-joinings of order $2$ if every ergodic two-fold self-joining $\lambda$ is either $\mu \times \mu$ or the image of $\mu$ under the map $x \rightarrow (x, T_g x)$ for some $g \in G$.
 
  We denote by $C(\bold{X})$ the centralizer of a measure preserving system $\bold{X} = (X, \mathcal{X}, \mu, T_G)$, that is, the set of all measure preserving transformations commuting with $T_g$ for all $g \in G$. It is known \cite{Ru1} that an ergodic system with minimal self-joinings has the somewhat unusual property that it has trivial centralizer and no proper factor. Theorems \ref{mild cor} and \ref{mild R} below show that a weak mixing measure preserving $\mathbb{Z}$ or $\mathbb{R}$ system with minimal self-joinings is mildly mixing. 
 
In this paper our main result is the following.
\begin{Theorem} [cf. Theorem \ref{msj theorem}, \ref{msj tiling}, \ref{eta}]
\label{main}
Let  $\theta$ be the substitution $\theta(0) = 001,$ $\theta(1) = 11001$ and let $\eta$ be the substitution with $\eta(0) = 001$, $\eta(1) = 11100$.
 \begin{enumerate}[(i)]
\item The substitution dynamical systems associated with $\theta$ and $\eta$ have minimal self-joinings, and hence are mildly mixing.
\item The substitution tiling spaces arising from the substitutions $\theta$ and $\eta$ using two intervals of irrational ratio have minimal self-joinings, and hence are mildly mixing.
 \end{enumerate}
\end{Theorem}

The first example of a substitution system with minimal self-joinings was Chacon's system; this was shown by del Junco, Rahe, and Swanson in \cite{JRS}. In \cite{JP}, del Junco and Park constructed examples of 1-dimensional $\mathbb{R}$-flows with minimal self-joinings based on Chacon's substitution, which can be viewed as 1-dimensional mildly mixing tiling dynamical systems. 
In this paper we employ methods similar to those used in \cite{JRS} and \cite{JP} to prove Theorem \ref{main}. 

Here is a brief description of the content of this paper.

In Section 2 we present basic notions and properties of joinings of dynamical systems.

In Section 3 we consider the substitution $\theta: 0 \rightarrow 001, 1 \rightarrow 11001$. First we prove that the substitution dynamical system associated with $\theta$ has minimal self-joinings and is mildly mixing. Then we show that the tiling dynamical system arising from the substitution $\theta$ and two intervals of irrational ratio has minimal self-joinings and is mildly mixing.

In Section 4 we investigate the substitution $\eta: 0 \rightarrow 001, 1 \rightarrow 11100$. We show that the substitution dynamical system associated with $\eta$ and the tiling dynamical system  arising from the substitution $\eta$ and two intervals of irrational ratio have minimal self-joinings and are mildly mixing.

Finally, in Section 5 we present an example of a tiling dynamical system which is rigid and weakly mixing. This example is constructed based on the dynamical system investigated by del Junco and Rudolph in \cite{JR}. 

\section{Joinings}

Let $\bold{X} = (X, \mathcal{X}, \mu, T_G)$ and $\bold{Y} = (Y, \mathcal{Y}, \nu, S_G)$ be two ergodic systems. A {\it{joining}} of $\bold{X}$ and $\bold{Y}$ is a probability measure $\lambda$ on $X \times Y$ such that (1) $\lambda$ is invariant under the diagonal action $T_G \otimes S_G= (T_g \times S_g)_{g \in G}$ and (2) $\lambda$ has marginals $\mu$ on $X$ and $\nu$ on $Y$: $\lambda(A \times Y) = \mu(A)$ and $\lambda(X \times B) = \nu(B)$ for any $A \in \mathcal{X}$ and $B \in \mathcal{Y}$. Let $J(\bold{X} , \bold{Y})$ be the set of all joinings of $\bold{X}$ and $\bold{Y}$. 
When $\bold{Y} = \bold{X}$, joinings are called {\it{self-joinings (of order 2)}}.  In this case $J(\bold{X} , \bold{X})$ is denoted by $J(\bold{X})$. Given $k \in \mathbb{N}$, one can define a $k$-fold self-joining $\lambda$ of an ergodic system $\bold{X}$ as a probability measure on $X^k$, which is invariant under $T_G \otimes \cdots \otimes T_G$ and whose marginals on each $X$ are $\mu$.

The centralizer, $C(\bold{X})$, is the set of all invertible measure preserving maps of $\bold{X}$ that commute with $T_g$ for all $g \in G$.  For $S_1, \dots , S_k \in C(\bold{X})$, the corresponding off-diagonal measure is the image of $\mu$ under the map $x \rightarrow (S_1(x), \cdots , S_k(x))$ of $X$ into $X^k$.
An ergodic system $\bold{X} = (X, \mathcal{X}, \mu, T_G)$ is called {\it{simple of order $k$}} (or {\it{$k$-simple}}) if any $k$-fold ergodic self-joining  $\lambda$ of $\bold{X}$ is a product of off-diagonal measures: there exists a partition of $J_1, \dots J_m$ of $\{1,2, \cdots, k\}$ such that for each $J_l$ the projection of $\lambda$ on $\prod_{i \in J_l } X$ is off-diagonal and $\lambda$ is the product of these off-diagonal measures.  We say that $\bold{X}$ has minimal self-joinings of order $k$ if $\bold{X}$ is $k$-simple with $C(\bold{X}) = \{ T_g : g \in G \}$.     
$\bold{X}$ is called simple or has minimal self-joinings (MSJ) of all orders if it is simple or has minimal self-joinings of order $k$ for any $k \in \mathbb{N}$ respectively.  

 J. King showed in (\cite{Ki2}) that 4-fold simple systems and 4-fold minimal self-joinings systems 
 are simple of all orders and minimal self-joinings of all orders respectively. Later Glasner, Host, and Rudolph improved King's result by showing that if $(X, \mathcal{X}, \mu, T)$ is weak-mixing and 3-simple, then it is simple.

 An action $(T_g)_{g \in G}$ of a group $G$ is called {\it{totally weakly mixing}} if each infinite subaction contains a mixing sequence $(T_{t_i})$, that is, $\lim\limits_{t_i \rightarrow \infty} \int f(x) g(T_{t_i} x) \, d \mu \rightarrow \int f \, d \mu \int g \, d \mu$. The following result is due to Ryzhikov.
  \begin{Theorem}[\cite{Ry} Theorem 2]
\label{R}
Let a 2-simple action $(T_g)_{g \in G}$ of a countable Abelian group be totally weakly mixing but not strongly mixing. In this case, $(T_g)_{g \in G}$ is simple of any order. 
\end{Theorem}
If $T$ is a weak mixing transformation, then $(T^n)_{n \in \mathbb{Z}}$ is totally weak mixing.
Hence if $(X, \mathcal{X}, \mu, T)$ has two-fold minimal self-joinings and is weakly mixing, but not strongly mixing, then it has minimal self-joinings of all orders. In fact, the same result works for $G= \mathbb{R}$.  

\begin{Corollary}
\label{cor R}
If $(X,\mathcal{X}, \mu, (T_t)_{(t \in \mathbb{R})})$ is weakly mixing  and not strongly mixing and has two-fold minimal self-joinings, then the system has minimal self-joinings of all orders.
\end{Corollary}

\begin{proof}
If $(T_t)_{(t \in \mathbb{R})}$ is a weak mixing $\mathbb{R}$-flow on $(X,\mathcal{X}, \mu)$, then $T_{\alpha}$ is a weak mixing transformation for any $\alpha \ne 0$. From Theorem \ref{R}, $\bold{\tilde{X}} = (X,\mathcal{X}, \mu, (T_t)_{(t \in \mathbb{Q})})$ is simple of all orders  and $C(\bold{\tilde{X}}) = \{T_t : t \in \mathbb{R} \}$. Hence $(X,\mathcal{X}, \mu, (T_t)_{(t \in \mathbb{R})})$ has minimal self-joinings of all orders.
 \end{proof}

Now we present a mixing property of minimal self-joinings of order 2.
\begin{Theorem}[see \cite{G} Corollary 12.5  and \cite{Ru} Theorem 6.12] 
\label{msj property}
Let $\bold{X} = (X, \mathcal{X}, \mu, T)$ be an invertible, non-atomic dynamical system with twofold minimal self-joinings. Then,
\begin{enumerate}[(i)]
\item $C(\bold{X}) = \{T^n : n \in \mathbb{Z} \}$.
\item $(X, \mathcal{X}, \mu, T)$ is prime, i.e., $(X, \mathcal{X}, \mu, T)$ does not have a proper factor.
\item $(X, \mathcal{X}, \mu, T)$ is weakly mixing.
\end{enumerate}
\end{Theorem}
 Theorem \ref{msj property} $(iii)$ can be strengthened as follows:
\begin{Theorem}
\label{mild cor}
Let $\bold{X} = (X, \mathcal{X}, \mu, T)$ be an invertible, non-atomic dynamical system with twofold minimal 
self-joinings. Then it is mildly mixing.
\end{Theorem}
\begin{proof}
Suppose that $ T^{n_i} f \rightarrow f$ in $L^2$ for some non-constant $f$. Then there exists a non-trivial rigid factor. From $(ii)$ of Theorem \ref{msj property}, $(X, \mathcal{X}, \mu, T)$ is rigid. It is well-known that if $(X,\mathcal{X}, \mu)$ is non-atomic and $\bold{X}$ is rigid, then $C(\bold{X})$ is uncountable (see \cite{KSS} or \cite{Ki}), which contradicts $(i)$ of Theorem \ref{msj property}. 
\end{proof} 

One can easily show that a result similar to Theorem \ref{msj property} holds for $\mathbb{R}$-actions:
\begin{Theorem}
\label{msj property R}
Let $\bold{X}=(X,\mathcal{X}, \mu, (T_t)_{t \in \mathbb{R}}) $ be an invertible, non-atomic dynamical system with twofold minimal self-joinings. Then,
\begin{enumerate}[(i)]
\item $C(\bold{X}) = \{T_t : t \in \mathbb{R} \}$.
\item $(X, \mathcal{X}, \mu, (T_t)_{t \in \mathbb{R}})$ is prime.
\end{enumerate}
\end{Theorem}

 It is not hard to check that the transitive flow $T_t x = x+ t $ mod $1$ on $[0,1)$ has minimal self-joinings but is not weakly mixing. Note that in this example, $(T_t)_{t \in \mathbb{R}}$ is not free; $T_n = Id$ for all $n \in \mathbb{Z}$. The following theorem shows that if the action of $\mathbb{R}$ is free, the situation is different.
 
\begin{Theorem}
\label{mild R}
Let $\bold{X} = (X,\mathcal{X}, \mu, (T_t)_{t \in \mathbb{R}}) $ be a dynamical system with minimal self-joinings. If $T_t \ne Id$ for all $t \ne 0$, then $\bold{X}$ is mildly mixing.
\end{Theorem}
\begin{proof}
By Theorem \ref{msj property R} $(i)$, the topological groups $\mathbb{R}$ and $C(\bold{X})$ are isomorphic. Suppose that there exists a non-constant rigid function in $L^2(X)$. Then $\bold{X}$ is rigid from Theorem \ref{msj property R} $(ii)$. Hence $T_{t_k} \rightarrow Id$ for some increasing sequence $t_k$, which is impossible in $\mathbb{R}$. 
\end{proof}

\section{the substitution $\theta: 0 \rightarrow 001, 1 \rightarrow 11001$}
 
 \subsection{The Substitution Dynamical System}\mbox{}
 
 In this section we will consider the substitution dynamical system arising from the substitution $\theta$.  As we have mentioned above, it was proved in \cite{DK} that this system is weakly mixing but not strongly mixing. We will strengthen this result by proving the following theorem.
\begin{Theorem}
\label{msj theorem}
The substitution dynamical system arising from the substitution $\theta: 0 \rightarrow 001, 1 \rightarrow 11001$ has minimal self-joinings and is mildly mixing.
\end{Theorem}  
  
   Before proving this result, let us present basic concepts and results of a primitive substitution system $(X_{\zeta}, \mathcal{X}, \mu, T)$ arising from a substitution $\zeta$.
Given a word $u = u_0 u_1 \cdots u_n$, let 
 $$[u] = [u_0 u_1 \cdots u_n] := \{ x \in X_{\zeta} : x_0 = u_0, x_1=u_1, \dots, x_n=u_n \}.$$
These sets and their translates are called {\it{cylinder sets}}, which are clopen and span the topology of $X_{\zeta}$. 
It is known that if $\zeta$ is primitive, $(X_{\zeta}, \mathcal{X}, \mu, T)$ is uniquely ergodic (\cite{Q}), that is, $\mu$ is the only $T$ invariant probability measure on $X_{\zeta}$. In this case we have a strong version of Ergodic Theorem: 
\begin{Theorem}  [see \cite{Ox} or \cite{Fu}]
\label{unique}
$(X, \mathcal{B}, \nu, T)$ is uniquely ergodic if and only if for every $f \in C(X)$,
\begin{equation}
\label{unique erg}
\lim_{N \rightarrow \infty} \frac{1}{N} \sum_{n=1}^N f(T^n x) = \int_X f \, d \nu
\end{equation}
uniformly on $X$.
\end{Theorem}
  A point $x \in X$ is called {\it{$\nu$-generic}}, if $x$ satisfies (\ref{unique erg}) for any $f \in C(X)$. If $\nu$ is ergodic measure on the topological dynamical system $(X,T)$, then $\nu$-almost all points are generic with respect to $\nu$.   
  
If $f$ is the characteristic function of a cylinder set for the word $u$ in a substitution space $X_{\zeta}$, then $\sum_{n=1}^N f(T^n x)$ counts the number of occurence of $u$ at the position $i$ in $x$ for $1 \leq i \leq N$. Thus the probability measure of the cylinder set $[u]$ is given by the occurrence frequency of $u$ in $x$ for $x \in X_{\zeta}$.   
 
 Let us introduce now some additional notation, which will be used throughout the Section 3.1. Define $n$-blocks $A_n = \theta^n (00)$ and $B_n = \theta^n(1)$ for $n \in \mathbb{N}$. Note that for any $n$, the finite word obtained by deleting first two letters in $A_n$ is the same as the word obtained by deleting the first letter in $B_n$. Denote this word as $C_n$, so we have $A_n = 00 C_n$ and $B_n = 1 C_n.$
Also notice that $A_{n+1} = A_n B_n A_n B_n$, $B_{n+1} = B_n B_n A_n B_n$, and the block $C_n$ begins with $C_{n-k}$ for $1 \leq k < n$. Denote the length of the word or block $w$ by $l=l(w)$. Let $l_n$  be the length of the word $C_n$, so $l_n = l(C_n)$. 

The following lemma says that for each $n \in \mathbb{N}$, $x \in X_{\theta}$ can be uniquely written in terms of $\theta^{n}(00)$ and $\theta^{n}(1)$.
\begin{Lemma}
\label{rep}
There is $m \in \mathbb{N}$ such that any admissible word $W$ of $X_{\theta}$ with $l(W) \geq m$ has the following unique expression
\begin{equation}
\label{rep eqn}
W = K_1 v_1 C_1 v_2 C_1 \cdots v_k C_1 K_2,
\end{equation}
where $K_1$ is a suffix of $\theta(00)$ or $\theta(1)$, $K_2$ is a prefix of $\theta(00)$ or $\theta(1)$ and $v_i = 00$ or $1$.
\end{Lemma}
\begin{proof}
It is obvious that any admissible word $W$ can be represented by formula (\ref{rep eqn}). Let us show that expression (\ref{rep eqn}) is unique.
 Since $\theta$ is primitive there exists $m \in \mathbb{N}$ such that any admissible word $W = w_1 w_2 \cdots w_l$ with $l \geq m$ contains $11001$, and so, there is $i$, $(1 \leq i \leq l-4)$ such that
$w_i w_{i+1} \cdots w_{i+4} = 11001.$
 If (\ref{rep eqn}) is not unique, there is $j$ with $1 \leq j \leq 3$ such that $w_{i+j} \cdots w_{i+4}$ is a prefix of $\theta(00)$ or $\theta(1)$. This is possible only for $j=2$ and in this case the first two letters $11$ of $11001$ is a suffix of $\theta(00)$ or $\theta(1)$, which is a contradiction.
\end{proof}

In \cite{JRS}, the Structure Lemma (Lemma 1 in \cite{JRS}), obtained by the cutting and stacking method, plays a crucial role in showing that the Chacon system has minimal self-joinings. In this paper we obtain the following similar result  based on Lemma \ref{rep}.  
\begin{Lemma}
\label{orbit}
Let $m$ be a positive integer as in Lemma \ref{rep}.
If $x$ and $y$ $\in X_{\theta}$ are not in the same orbit, then for infinitely many $n$, there exist $m_{1,n}$ and $m_{2,n}$ $\in \mathbb{Z}$ such that 
\begin{enumerate}[(i)]
\item $|m_{i,n}| \leq (m+3)( l_n + 2 )$ for $i=1,2$ and $|m_{1,n} - m_{2,n}| \leq \frac{1}{2} (l_n + 3)$,
\item either $C_n 00 C_n $ occurs at $m_{1,n}$ of $x$ and $C_n 1 C_n$ occurs at $m_{2,n}$ of $y$, 
or $C_n 1 C_n $ occurs at $m_{1,n}$ of $x$ and $C_n 00 C_n$ occurs at $m_{2,n}$ of $y$.
\end{enumerate}
\end{Lemma}

\begin{proof}
For a given $n_0$, we will find an integer $n \geq n_0$ satisfying the condition stipulated in the formulation. 

Let us express $x$ and $y$ in terms of the $n_0$-blocks $A_{n_0} = \theta^{n_0}(00)$ and $B_{n_0}=\theta^{n_0}(1)$. Then we can find an integer $k$ ($|k| \leq \frac{1}{2} (l_{n_0}+2)$) such that the block containing $x_0$ begins at the same place of the  block containing $(T^k y)_0$. Introduce the sequence of $n_0$-blocks $(D_i^{n_0})_{i \in \mathbb{Z}}$ such that $x= \cdots D_{-1}^{n_0} D_0^{n_0} D_1^{n_0} \cdots$, where $D_i^{n_0}$ is $A_{n_0}$ or $B_{n_0}$ and $x_0$ belongs to $D_0^{n_0}$. Let $\delta_i$ be the position where $D_i^{n_0}$ occurs in $x$. Similarly, $T^ky$ can be written in terms of $n_0$-blocks $(E_i^{n_0})_{i \in \mathbb{Z}}$ and $\epsilon_i$ denotes the position where $E_i^{n_0}$ occurs in $T^k y$. Note that $\delta_0 = \epsilon_0$. Since $x$ and $y$ are not in the same orbit, there exists $s$ such that $D_s^{n_0}$ and $E_s^{n_0}$ are different. We choose $s_{n_0}$ such that $|s_{n_0}|$ is minimal for such $s$. Without loss of generality we can assume that $D_{s_{n_0}}^{n_0} = B_{n_0} $ and $E_{s_{n_0}}^{n_0}=A_{n_0}$.

Suppose $|s_{n_0}| > m$. We can express $x$ and $T^k y$ in terms of the $(n_0 + 1)$-blocks. Note that 
two different $(n_0+1)$-blocks occur at the positions in $x$ and $T^ky$ where two different $n_0$-blocks $D_{s_{n_0}}^{n_0}$ and $E_{s_{n_0}}^{n_0}$ occur. Since  the length of $n$-blocks is increasing, $|s_{n_0 +1}| < |s_{n_0}|$.
 We can find the minimal $n \geq n_0$ such that $|s_n| \leq m$. Now we consider the two cases $s_{n} \geq 0$ and $s_{n} < 0$. If $s_{n} \geq 0$, then $C_{n}1 C_{n}$ occurs at position $\delta_{s_{n}} - l_{n}$ in $x$ and $C_{n} 00 C_{n}$ occurs at position $\delta_{s_{n}} - l_{n}$ in $T^ky$ (see figure 1.) Let $m_{1, n} = \delta_{s_{n}} - l_{n}$  and $m_{2, n} = \delta_{s_{n}} - l_{n} -k$. Then 
\begin{enumerate}[(a)]
\item $|m_{1,n}| \leq |\delta_{s_{n}}| + |l_{n}| \leq (m+1)(l_{n}+2) + l_{n} \leq (m+2)(l_{n} +2)$. 
\item $|m_{2,n}| \leq |\delta_{s_{n}}| + |l_{n}| + |k| \leq (m+1)(l_{n}+2) + l_{n} + \frac{1}{2}(l_{n}+2) \leq (m+3)(l_{n} +2)$.
\item $|m_{1,n} - m_{2,n}| \leq |k| \leq  \frac{1}{2}(l_{n}+2)$.
\end{enumerate}
If $s_{n} < 0$, then $C_{n}1 C_{n}$ occurs at $\delta_{s_{n}} - l_{n}$ in $x$ and $C_{n} 00 C_{n}$ occurs at $\delta_{s_{n}} - l_{n} - 1$ in $T^ky$. Let $m_{1, n} = \delta_{s_{n}} - l_{n}$  and $m_{2, n} = \delta_{s_{n}} - l_{n} -k -1$. Then by the similar consideration we have $|m_{i,n}| \leq (m+3)(l_{n}+2)$ and $|m_{1,n} - m_{2,n}| \leq |k| +1 \leq  \frac{1}{2}(l_{n}+3)$.
 \end{proof}

\setlength{\unitlength}{.4in}
\begin{picture}(5,1)(0,0)
\label{fig 1}
\linethickness{1pt}
\put(0,0){\line(1,0){15}}
\put(-.5, 0){\makebox(0,0){$x:$}}
\put(2,-.3){\makebox(0,0){$x_0$}}
\put(5.1,-.1){\line(0,1){0.2}}
\put(7,-.3){\makebox(0,0){$C_n$}}
\put(8.9,-.1){\line(0,1){0.6}}
\put(9.1,-.1){\line(0,1){0.2}}
\put(9,- .3){\makebox(0,0){$1$}}
\put(11,-.3){\makebox(0,0){$C_n$}}
\put(8.9,0.3){\line(1,0){4.1}}
\put(11, 0.7){\makebox(0,0){$D_{s_n}^n$}}
\put(13,-.1){\line(0,1){0.6}}
\put(0,-2){\line(1,0){15}}
\put(-.7, -2){\makebox(0,0){$T^ky:$}}
\put(5.1,-2.1){\line(0,1){0.2}}
\put(7,-2.3){\makebox(0,0){$C_n$}}
\put(8.9,-2.1){\line(0,1){0.6}}
\put(9.1,-2.1){\line(0,1){0.2}}
\put(9.3,-2.1){\line(0,1){0.2}}
\put(9, -2.3){\makebox(0,0){$0$}}
\put(9.2, -2.3){\makebox(0,0){$0$}}
\put(11.2,-2.3){\makebox(0,0){$C_n$}}
\put(13.2,-2.1){\line(0,1){0.6}} 
\put(8.9,-1.7){\line(1,0){4.3}}
\put(11, -1.4){\makebox(0,0){$E_{s_n}^n$}}

\end{picture}

\vspace{2cm}

\setlength{\unitlength}{.4in}
\begin{picture}(5,3)(0,0)
\linethickness{1pt}
\put(0,0){\line(1,0){15}}
\put(-.5, 0){\makebox(0,0){$x:$}}
\put(2,-.3){\makebox(0,0){$x_0$}}
\put(5.1,-0.1){\line(0,1){0.2}}
\put(5.1, -.3){\makebox(0,0){$m_{1,n}$}}
\put(7, -0.3){\makebox(0,0){$C_n$}}
\put(8.9,-.1){\line(0,1){0.2}}
\put(9.1,-.1){\line(0,1){0.2}}
\put(9,- .3){\makebox(0,0){$1$}}
\put(11,-.3){\makebox(0,0){$C_n$}}
\put(13,-.1){\line(0,1){0.2}}
 \put(0,-2){\line(1,0){15}}
\put(-.5, -2){\makebox(0,0){$y:$}}
\put(2,-2.5){\makebox(0,0){$y_0$}}
\put(4.1,-2.3){\makebox(0,0){$m_{2,n}$}}
\put(4.1,-2.1){\line(0,1){0.2}}
\put(6,-2.3){\makebox(0,0){$C_n$}}
\put(7.9,-2.1){\line(0,1){0.2}}
\put(8.1,-2.1){\line(0,1){0.2}}
\put(8.3,-2.1){\line(0,1){0.2}}
\put(8, -2.3){\makebox(0,0){$0$}}
\put(8.2, -2.3){\makebox(0,0){$0$}}
\put(10.2,-2.3){\makebox(0,0){$C_n$}}
\put(12.2,-2.1){\line(0,1){0.2}} 

\put(7, -3.5){\makebox(0,0){Figure 1}}
\end{picture}

\vspace{4cm}

To show that $(X_{\theta}, \mathcal{X}, \mu, T)$ has minimal self-joinings, we will use the following two technical lemmas.    
\begin{Lemma}[\cite{Ru} Lemma 6.14]
\label{disjoint lemma}
Let $I$ be the identity map on $(X, \mathcal{X}, \nu_1)$ and $S$ be an ergodic map on $(Y, \mathcal{Y}, \nu_2 )$.
If $\overline{\nu}$ is a joining of $(X, \mathcal{X}, \nu_1, I)$ and $(Y, \mathcal{Y}, \nu_2, S)$, then $\overline{\nu} = \nu_1 \times \nu_2$.
\end{Lemma}

\begin{Lemma}[\cite{Ru} Lemma 6.15]
\label{msj lemma}
 Let $\bold{X} = (X, \mathcal{X}, \mu, T)$ be an ergodic system and $\{P_i \}$ be a countable set of cylinder sets generating $\mathcal{X}$. Let $\overline{\mathcal{A}} = \{P_l \times P_m\}$ be a countable generating algebra of $X \times X$.
 Assume that
 \begin{enumerate}
\item $\overline{\mu} \in J(\bold{X})$ is a two-fold ergodic joining.
\item $(x,y) \in X \times X$ satisfies 
    $$\lim_{N \rightarrow \infty} \frac{1}{N} \sum_{n=0}^{N-1} 1_A(T^{-i}x, T^{-i} y) = \lim_{N \rightarrow \infty} \frac{1}{N} \sum_{n=0}^{N-1} 1_A(T^{i}x, T^{i} y) = \overline{\mu}(A)$$
for all $A \in \overline{\mathcal{A}}$.
\item There are intervals $L_k = [i_k, j_k] \subset \mathbb{Z}$,  $M_k = [a_k, b_k]$, $t_k \in \mathbb{Z}$ and $\gamma > 0$ such that
    $i_k \leq 0 \leq j_k$ and $j_k - i_k \rightarrow \infty$, $M_k \subset L_k$ and $t_k + M_k \subset L_k$, and $|M_k| \geq \gamma |L_k| $.
\item For any cylinder sets $P_l, P_m$, there exists $K=K(P_l, P_m)$ such that if $k \geq K$, then for all $i \in M_k$,
 \subitem$T^i x \in P_l$ if and only if $T^{t_k + i}x \in P_l$, 
 \subitem$T^{i} y \in P_m$ if and only if $T^{t_k + i +1}y \in P_m $.
 \end{enumerate}  
Then $\overline{\mu}$ is an $I \times T$ invariant measure on $X \times X$and hence $\overline{\mu} = \mu \times \mu$.
\end{Lemma}

\begin{proof}[Proof of Theorem \ref{msj theorem}]
Note that the substitution dynamical system $(X_{\theta}, \mathcal{X}, \mu, T)$ arising from $\theta$ is weakly mixing, hence $\mu$  is non-atomic. We know from \cite{DK} that it is not strongly mixing. Hence if we show $(X_{\theta}, \mathcal{X}, \mu, T)$ has two-fold minimal self-joinings, then it will have minimal self-joinings of all orders by Corollary \ref{cor R}  and will be mildly mixing by Theorem \ref{mild cor}. 

 Given an ergodic joining $\overline{\mu}$, we can find a $\overline{\mu}$-generic point $(x,y)$ satisfying condition (2) in Lemma \ref{msj lemma}. If $x$ and $y$ are in the same orbit, that is $T^k x = y$ for some $k \in \mathbb{Z}$, then $\overline{\mu}$ is an off-diagonal measure which is the image of $\mu$ under the map $z \rightarrow (z,T^kz)$.  
 
Otherwise, introduce a new alphabet $\mathcal{A}_s$ consisting of the letters $h_i^s$, where $C_s = h_1^s \cdots h_{l_s}^s$ for each $s$. (Recall that $C_s$ is a block such that $\theta^s(00)= 00C_s$.) Let $P_i^s$ be the cylinder for the alphabet $h_i^s$. Then $\{P_i^s : 1 \leq s \leq l_s, i \in \mathbb{N} \}$ is a countable set of cylinder sets generating $\mathcal{X}$.

By Lemma \ref{orbit}, without loss of generality we assume that there exists an increasing sequence of integers $(n_k)_{k=1}^{\infty}$ such that the two blocks $C_{n_k} 1 C_{n_k}$ and $C_{n_k} 00 C_{n_k}$ appear in $L_k=[-(m+4)(l_{n_k}+2), (m+4)(l_{n_k}+2)]$ of $x$ and $y$ respectively (see figure 2). Let $M_k$ be the interval in the first $C_{n_k}$ blocks where $C_{n_k}1 C_{n_k}$ and $C_{n_k} 00 C_{n_k}$ overlap. Its length satisfies $|M_k| \geq |C_{n_k}| - |m_{1,n_k} - m_{2, n_k}| \geq\frac{1}{2} l_{n_k} -3 $. Let $t_k=l_{n_k}+1$. Then for $s \geq n_k$, we have
$$T^i x \in P_l^s  \Leftrightarrow T^{t_k + i}x \in P_l^s \,\,\,\, \textrm{and} \,\,\,\, 
T^{i} y \in P_m^s \Leftrightarrow T^{t_k + i +1}y \in P_m^s$$
for any $1 \leq l,m \leq h_{l_s}$
Also, we have
 $$ \liminf_{k \rightarrow \infty} \frac{|M_k|}{2(m+4)(l_{n_k}+2) } \geq \liminf_{k \rightarrow \infty} \frac{l_{n_k} -6 }{4(m+4)(l_{n_k}+2)} =  \frac{1}{4(m+4)} > 0.$$  
Lemma \ref{msj lemma} implies that $\overline{\mu}$ is $I \times T$ invariant and thus, $\overline{\mu} = \mu \times \mu$.
\end{proof}
 
 \setlength{\unitlength}{.4in}
\begin{picture}(5,3)(0,-1.5)
\linethickness{1pt}
\put(0,0){\line(1,0){15}}
\put(-.5, 0){\makebox(0,0){$x:$}}
\put(2,.3){\makebox(0,0){$x_0$}}

\put(5.1,-0.1){\line(0,1){0.2}}
\put(5.1, .3){\makebox(0,0){$m_{1,n_k}$}}
\put(7, 0.3){\makebox(0,0){$C_{n_k}$}}
\put(8.9,-.1){\line(0,1){0.2}}
\put(9.1,-.1){\line(0,1){0.2}}
\put(9,.3){\makebox(0,0){$1$}}
\put(11, .3){\makebox(0,0){$C_{n_k}$}}
\put(13,-.1){\line(0,1){0.2}}
 \put(0,-2){\line(1,0){15}}
\put(-.5, -2){\makebox(0,0){$y:$}}
\put(2,-2.3){\makebox(0,0){$y_0$}}
\put(4.1,-2.3){\makebox(0,0){$m_{2,n_k}$}}
\put(4.1,-2.1){\line(0,1){0.2}}
\put(6,-2.3){\makebox(0,0){$C_{n_k}$}}
\put(7.9,-2.1){\line(0,1){0.2}}

\put(8.1,-2.1){\line(0,1){0.2}}
\put(8.3,-2.1){\line(0,1){0.2}}
\put(8, -2.3){\makebox(0,0){$0$}}
\put(8.2, -2.3){\makebox(0,0){$0$}}
\put(10.2,-2.3){\makebox(0,0){$C_{n_k}$}}
\put(12.2,-2.1){\line(0,1){0.2}} 


\put(5.1,-0.4){\line(1,0){2.8}}
\put(5.1,-0.4){\line(0,1){0.4}}
\put(7.9,-0.4){\line(0,1){0.4}}
\put(9.1,-0.4){\line(1,0){2.8}}
\put(9.1,-0.4){\line(0,1){0.4}}
\put(11.9,-0.4){\line(0,1){0.4}}

\put(5.1,-1.5){\line(1,0){2.8}}
\put(5.1,-2){\line(0,1){0.5}}
\put(7.9,-2){\line(0,1){0.5}}

\put(9.4,-1.5){\line(1,0){2.8}}
\put(9.4,-2){\line(0,1){0.5}}
\put(12.2,-2){\line(0,1){0.5}}
\put(6.2, -.3){\makebox{$M_k$}}
\put(9.6, -.3){\makebox{$t_k + M_k$}}
\put(6.2, -1.9){\makebox{$M_k$}}
\put(9.7, -1.9){\makebox{$t_k + 1 + M_k$}}

\put(7, -3.5){\makebox(0,0){Figure 2}}

\end{picture}

\vspace{4cm}



\subsection{Tilling Spaces}\mbox{}

Let us present some basic facts on the substitution tiling system arising from the substitution $\zeta$ on the alphabet $\{0,1\}$ and intervals $J_0$ and $J_1$. Let $\mathcal{T}$ be a substitution tiling arising from $\zeta$ and  intervals $J_0$ and $J_1$ and let $X_{\mathcal{T}}$ be the corresponding tiling space. For any $\zeta$-admissible word  $u= u_0 u_1 \cdots u_n$ and any interval $I \subset [0, |J_{u_0}|)$, define 
 \begin{equation*}
[u] \times I :=  \{ \mathcal{S} \in X_{\mathcal{T}} :  \textrm{for some} \, t \in I, \,  \,  \, (\mathcal{S} - t) \,\, \textrm{and} \,\,J_{u_0}J_{u_1} \cdots J_{u_n} \,\, \textrm{agree on} \,\,  [0, |J_{u_0}| + \cdots + |J_{u_n}| )  \}. 
\end{equation*}
These sets and their translates are called cylinder sets in the tiling space. They are clopen sets and span the topology of the tiling space.
Another description of the $\mathbb{R}$-flow on the substitution tiling space is as a flow under the function $f$ built over the substitution subshift $(X_{\zeta},T)$, where $f: X_{\zeta} \rightarrow \mathbb{R}$ is defined by $f(x) = |J_0|$ for $x \in [0]$ and $f(x) = |J_1|$ for $x \in [1]$.

It is known that the tiling system is minimal and has a unique invariant probability measure $\nu$ if the substitution is primitive. Then we have
\begin{equation}
\label{tiling measure}
 \nu ([u] \times I) = \mu([u]) \times \frac{ |I| }{\mu([0]) |J_0| + \mu([1]) |J_1| },
 \end{equation}
 where $\mu$ is the unique invariant probability measure on the substitution system arising from $\zeta$ (see Lemma 2.1 in \cite{CS}). 
Clark and Sadun obtained the following condition for this class of tiling systems to be weakly mixing.
\begin{Theorem}[cf \cite{CS} Theorem 2.5]
 \label{CS}
 Let $\zeta$ be an aperiodic, primitive, constant length substitution on the alphabet $\mathcal{A} = \{ 0,1\}$, where the number of $0$s occurring in $\zeta(0)$ is different from the number of $0$s occurring in $\zeta(1)$. Then the substitution tiling system arising from $\zeta$ and intervals $J_0$ and $J_1$, where $\frac{|J_0|}{|J_1|}$ is irrational, is weakly mixing.
 \end{Theorem}
  See \cite{Ro2} and \cite{S} for more information about tiling dynamical systems  and \cite{CS} for 1-dimensional substitution tiling spaces.

In this subsection we will consider the substitution tiling ${\mathcal{T}}$ of $\mathbb{R}$ arising from  the substitution  $\theta: 0 \rightarrow 001, 1 \rightarrow 11001$ and two intervals $J_0, J_1$ and the corresponding substitution tiling system $\bold{X_{\mathcal{T}}} = (X_{\mathcal{T}}, \mathcal{D}, \lambda, (T_t)_{t \in \mathbb{R}} )$. Our goal is to show that if $\frac{|J_0|}{|J_1|}$ is irrational, then the tiling system $\bold{X_{\mathcal{T}}}$ has minimal self-joinings and is mildly mixing.
  
By writing $a=00$ and $b=1$, we obtain a substitution  $\tilde{\theta}: a \rightarrow abab, b \rightarrow bbab$ from the substitution $\theta$. Let $\tilde{\mathcal{T}}$ be the substitution tiling arising from the substitution $\tilde{\theta}$ and intervals $J_a$ and $J_b$, and let $\bold{X_{\tilde{\mathcal{T}}}} = (X_{\tilde{\mathcal{T}}},  \tilde{\mathcal{D}}, \tilde{\lambda}, (T_t)_{t \in \mathbb{R}} )$ be the corresponding substitution tiling system. If $|J_a| = 2 |J_0|$ and $|J_b| = |J_1|$, then any tiling in $\bold{X_{\mathcal{T}}}$ uniquely corresponds to a tiling in $\bold{X_{\tilde{\mathcal{T}}}}$ by writing $J_0J_0 = J_a$ and $J_1=J_b$. From this observation, it is not hard to see that $\bold{X_{\tilde{\mathcal{T}}}}$ and $\bold{X_{\mathcal{T}}}$ are topologically conjugate if $|J_a| = 2 |J_0|$ and $|J_b| = |J_1|$. Since these systems have unique invariant probability measures, they are also measurably isomorphic. 
Our main result in this section is the following. (Note that by Theorem \ref{CS} the systems  
$\bold{X_{{\mathcal{T}}}}$ and $\bold{X_{\tilde{\mathcal{T}}}}$ appearing in Theorem \ref{msj tiling} are weakly mixing.) 
\begin{Theorem}
\label{msj tiling} 
If $\frac{|J_a|}{|J_b|}$ is irrational, then $\bold{X_{\tilde{\mathcal{T}}}} = (X_{\tilde{\mathcal{T}}},  \tilde{\mathcal{D}}, \tilde{\lambda}, (T_t)_{t \in \mathbb{R}} )$ has minimal self-joinings and is mildly mixing. Hence if $\frac{|J_0|}{|J_1|}$ is irrational, $\bold{X_{{\mathcal{T}}}}$ has minimal self-joinings and is mildly mixing. 
\end{Theorem}

The following lemma is the $\mathbb{R}$-flow version of Lemma \ref{msj lemma}. It is not hard to prove Lemma \ref{msj tiling lemma} using the argument in the proof of Lemma 6.15 in \cite{Ru}.
\begin{Lemma}
\label{msj tiling lemma}
 Let $\bold{X} = (X, \mathcal{X}, \mu, (T_t)_{t \in \mathbb{R}})$ be an ergodic system. Let $\alpha \ne 0$ such that $T_{\alpha}$ is an ergodic transformation. Let $\{P_i \}$ be a countable set of cylinder sets generating $\mathcal{X}$. Let $\overline{\mathcal{A}} = \{P_l \times P_m\}$ be a countable generating algebra of $X \times X$.
 Assume that
 \begin{enumerate}
\item $\overline{\mu} \in J(\bold{X})$ is a two-fold ergodic joining.
\item $(x,y) \in X \times X$ satisfies 
    $$\lim_{L \rightarrow \infty} \frac{1}{L} \int_{-L}^{0} 1_A(T_{t}x, T_{t} y) \, d t = \lim_{L \rightarrow \infty} \frac{1}{L} \int_{0}^{L} 1_A(T_t x, T_t y) \, dt = \overline{\mu}(A)$$
for all $A \in \overline{\mathcal{A}}$.
\item There are intervals $L_k = [i_k, j_k] \subset \mathbb{R}$,  $M_k = [a_k, b_k]$, $t_k \in \mathbb{R}$ and $\gamma > 0$ such that
    $i_k \leq 0 \leq j_k$ and $j_k - i_k \rightarrow \infty$, $M_k \subset L_k$ and $t_k + M_k \subset L_k$, and $|M_k| \geq \gamma |L_k| $.
\item For any cylinder sets $P_l, P_m$,there exists $K=K(P_l, P_m)$ such that if $k \geq K$, then for all $t \in M_k$,
\subitem $T_t x \in P_l$ if and only if $T_{t_k + t}x \in P_l$,
 \subitem $T_{t} y \in P_m$ if and only if $T_{t_k + t +\alpha} y \in P_m $.
 \end{enumerate}  
Then $\overline{\mu}$ is an $I \times T_{\alpha}$ invariant measure on $X \times X$ and hence $\overline{\mu} = \mu \times \mu$.
\end{Lemma}

\begin{proof}[Proof of Theorem \ref{msj tiling}]
We will assume that $|J_b| > |J_a|$. The proof for the other case is analogous.
For $\alpha= |J_b|-|J_a| $, $T_{\alpha}$ is a weak-mixing transformation on $(X_{\tilde{\mathcal{T}}},   \tilde{\mathcal{D}}, \tilde{\lambda})$, since $\bold{X_{{\tilde{\mathcal{T}}}}}$ is weakly mixing.

Let $\overline{\lambda}$ be a twofold ergodic joining of $\bold{X_{\tilde{\mathcal{T}}}}$  and let $(\mathcal{S}_1,\mathcal{S}_2)$ be a $\overline{\lambda}$-generic point satisfying the condition (2) in Lemma \ref{msj tiling lemma}.

 If $\mathcal{S}_1$ and $\mathcal{S}_2$ are on the same orbit under $(T_t)_{t \in \mathbb{R}}$, then $\overline{\lambda}$ is off-diagonal. Otherwise,
we can view $\mathcal{S}_1$ and $\mathcal{S}_2$ as the tilings arising from tiles $J_0$ and $J_1$ by writing $J_a = J_0 J_0$ and $J_b =J_1$. Then, for infinitely many $n$, patches corresponding to $C_n1C_n$ and $C_n00C_n$ occurs in $\mathcal{S}_1$ and $\mathcal{S}_2$ such that they occur not too far from the origin and their overlap is sufficiently large. By applying the same argument as in the proof of Theorem \ref{msj theorem} with Lemma \ref{msj tiling lemma}, $\overline{\lambda}$ is $I \times T_{\alpha} $ invariant, so $\overline{\lambda} = \tilde{\lambda} \times \tilde{\lambda}$. Hence $\bold{X_{\tilde{\mathcal{T}}}}$ has two-fold minimal self-joinings, and so it is mildly mixing. Moreover, it has minimal self-joinings of all orders since the substitution tiling system is not strongly mixing. 
\end{proof} 

\subsection{More examples}\mbox{}

The previous results can be generalized. Let $s$ be a primitive substitution of constant length on the alphabet $\{a,b\}$ with $s(a) = aA$ and $s(b) = bA$, where $A$ is a finite word over $\{a,b\}$. In this subsection we will show that a substitution tiling system arising from $s$ and two intervals $J_a, J_b$, where $\frac{|J_a|}{|J_b|}$ is irrational, has minimal self-joinings and so is mildly mixing.

Let us first show that this systems is weakly mixing. By Theorem \ref{CS}, it is enough to show that $s$ is an aperiodic substitution.  Suppose that $T^k(x) = x$ for some $x \in X_{s}$. There is a non-negative integer $n$ such that $l(s^{n}(a)) \leq k < l(s^{n+1}(a))$. Let $A_n$ be the block such that $s^n(a) = aA_n$ and $s^n(b)=bA_n$. Note that $A=A_1$. The blocks $a A_n$ and $b A_n$ both appear in $x$. Since $T^k(x) =x$, the $k$-th letter in $A_n$ is $a$ from $a A_n$ and the $k$-th letter in $A_n$ is $b$ from $b A_n$, which leads to a contradiction.
Hence, a tiling system arising from the substitution $s$ and two intervals  $J_a, J_b$, where $\frac{|J_a|}{|J_b|}$ is irrational, is weakly mixing.
We have the following lemma similar to Lemma \ref{orbit}.

\begin{Lemma}
\label{induced}
There is $m \in \mathbb{N}$ such that any admissible word $W$ of $X_s$ with $l(W) \geq m$ has the following unique expression
$$W = K_1 v_1 A v_2 A \cdots v_k A K_2,$$
where $K_1$ is  a suffix of $s(a)$ or $s(b)$, $K_2$ is a prefix of $s(a)$ or $s(b)$ and $v_i = a$ or $b$.
\end{Lemma}
\begin{proof}
Since $s$ is primitive, $aa$ or $bb$ is an admissible word. We may assume that $bb$ is admissible. Then $bba$ is also admissible, otherwise there exists a sequence $x= (x_n) \in X_s$ such that $x_n=b$ for all $n$. 

Choose $m$ such that any admissible word $W = w_1 w_2 \cdots w_l$ with $l \geq m$ contains $s(bba) = bA bA aA$. It is clear that we have the above expression for $W$. If there is another expression, then $\alpha A \beta A$ occurs in $bA bA aA $, where $\alpha$ occurs in the first $A$ of $bA bA aA$ and $\beta$ occurs in the second $A$ of $bA bA aA$.  Then the letter $b$ occurs at some position $i$ of the first $A$ of $\alpha A \beta A$ and the letter $a$ occurs in the same position $i$ of the second $A$ of $\alpha A \beta A$, which is a contradiction.
\end{proof}

Let  $(X, \mathcal{X}, \mu, T )$ be a subshift generated by $x \in \{0,1\}^{\mathbb{Z}}$, where $x$ is obtained from a sequence $y \in X_s$ by writing $a=00$ and $b=1$.  Let $C_n$ be the block of letters $0$ and $1$ obtained from $A_n$ by writing $a=00$ and $b=1$. Applying Lemma \ref{induced}, we can obtain a result similar to Lemma \ref{orbit} for  the the subshift $(X, \mathcal{X}, \mu, T )$ with an appropriate $m$. The methods used in Sections 3.1 and 3.2 imply that $(X, \mathcal{X}, \mu, T )$ has minimal self-joinings and that the tiling dynamical system arising from the substitution $s$ and intervals $J_0, J_1$, where $ \frac{|J_0|}{|J_1|}$ is irrational, has minimal self-joinings.  

\section{the substitution $\eta: 0 \rightarrow 001,1 \rightarrow 11100$}

Let $(X_{\eta}, \mathcal{X}, \mu, T )$ be the dynamical system arising from the substitution $\eta$ and two intervals $J_0, J_1$ of irrational ratio. Let $\tilde{\eta}$ be the substitution $a \rightarrow abab, b \rightarrow bbba$. Then the substitution tiling system arising from $\eta$ and two intervals $J_0, J_1$ is isomorphic to the substitution tiling system arising from $\tilde{\eta}$ and two intervals $J_a, J_b$ if $|J_a| = 2 |J_0|$ and $|J_b|=|J_1|$. In \cite{BR} it was shown that the substitution tiling system arising from $\tilde{\eta}$ and two intervals $J_a, J_b$, where $\frac{|J_a|}{|J_b|}$ is irrational, is weakly mixing. Thus, the substitution tiling system arising from ${\eta}$ and two intervals $J_0, J_1$, where$\frac{|J_0|}{|J_1|}$ is irrational, is weakly mixing. In this section we will strengthen these results. 

\begin{Theorem} 
\label{eta}
Let $\eta$ be a substitution with $\eta (0) = 001$ and $\eta(1)= 11100$.
\begin{enumerate}[(a)]
\item The substitution dynamical system $(X_{\eta}, \mathcal{X}, \mu, T )$ has minimal self-joinings and is mildly mixing.
\item The substitution tiling system arising from $\eta$ and intervals $J_0$ and $J_1$, where $\frac{|J_0|}{|J_1|}$ is irrational, has minimal self-joinings and is mildly mixing.
\end{enumerate}
 \end{Theorem}

Define $n$-blocks $A_n = \eta^n(00)$ and $B_n= \eta^n(1)$. Denote $l_n = l(A_n) = l(B_n)+1$. Note that $l_{n+1} = 4l_n+1$. The following lemma is similar to Lemma \ref{rep}.  

 \begin{Lemma}
 \label{Lemma4.1}
There is $m \in \mathbb{N}$ such that any admissible word $W$ of $X_{\eta}$ with $l(W) \geq m$ has the following unique expression
$$W = K_1  v_1 C_1 v_2 C_2 \cdots v_k C_k K_2,$$
where $K_1$ is a suffix of $\eta(00)$ or $\eta(1)$, $K_2$ is a prefix of $\eta(00)$ or $\eta(1)$, and $v_iC_i$ is $\eta(00)$ or $\eta(1)$.
\end{Lemma}
The proof for this lemma is just the proof of Lemma \ref{rep}, using the word $11100$ instead of $11001$. From this result, we have the following structure lemma.
\begin{Lemma}
\label{orbit2} Let $m$ be the positive integer in Lemma \ref{Lemma4.1}. 
If $x$ and $y$ are not in the same orbit, then for infinitely many $n$, there exists $m_1$ and $m_2$ with $|m_i| \leq (m +3) l_{n+1} $ and $|m_1 - m_2| \leq \frac{1}{2} l_{n} + 2$ such that one of the following occurs:
\begin{enumerate}[(i)]
\item $A_n B_n A_n$ and $B_n B_n B_n$ occur at $x_{m_1}$ and $y_{m_2}$ (or $y_{m_2}$ and $x_{m_1}$) respectively. 
\item $B_n A_n B_n$ and $B_n B_n B_n$ occur at $x_{m_1}$ and $y_{m_2}$ (or $y_{m_2}$ and $x_{m_1}$) respectively. 
\item $A_n A_n$ and $B_n B_n$ occur at $x_{m_1}$ and $y_{m_2}$ (or  $y_{m_2}$ and $x_{m_1}$) respectively. 
\end{enumerate}
\end{Lemma}

\begin{proof}
Given $n_0 \in \mathbb{N}$, $x$ and $y$ can be expressed in terms of $A_{n_0+1}=\eta^{n_0+1}(00)$ and $B_{n_0+1}=\eta^{n_0+1}(1)$.
Then we can find an integer $\tilde{k}$ ($|\tilde{k}| \leq \frac{1}{2} l_{n_0+1}$) such that the block containing $x_0$ begins at the same place of the  block containing $(T^{\tilde{k}} y)_0$. Introduce the sequence of $(n_0+1)$-blocks $(D_i^{n_0+1})_{i \in \mathbb{Z}}$ such that $x= \cdots D_{-1}^{n_0+1} D_0^{n_0+1} D_1^{n_0+1} \cdots$, where $D_i^{n_0+1}$ is $A_{n_0+1}$ or $B_{n_0+1}$ and $x_0$ belongs to $D_0^{n_0+1}$. Let $\delta_i$ be the position where $D_i^{n_0+1}$ occurs in $x$. Similarly, write $T^{\tilde{k}}y$ in terms of the $(n_0+1)$-blocks $(E_i^{n_0+1})_{i \in \mathbb{Z}}$ and let $\epsilon_i$ denote the position where $E_i^{n_0+1}$ occurs in $T^{\tilde{k}} y$. Since $x$ and $y$ are not in the same orbit, there exists $s$ such that $D_s^{n_0+1}$ and $E_s^{n_0+1}$ are different. We choose $s_{n_0+1}$ such that $|s_{n_0+1}|$ is minimal for such $s$. 
If $|s_{n_0+1}| > m+1$, we express $x$ and $T^{\tilde{k}}y$ as sequences of $(D_i^{n_0+2})$ and $(E_i^{n_0+2})$. In this case, $|s_{n_0+2}| < |s_{n_0+1}|$. Hence we can find an integer $n \geq n_0$ such that $|s_{n+1}| \leq m$. 

Let $k = \tilde{k}$ if  $s_{n +1} \geq 0$. If $s_{n +1} < 0$, then choose  $k = \tilde{k} \pm 1$ such that the two different $(n+1)$-blocks in $x$ and $T^ky$ occur at the same place.  
Hence, one of the following holds:
\begin{enumerate}
\item  $A_{n+1}$ occurs in $x$ and $B_{n+1}$ occurs in $T^k y$ at the same place $i$ with $|i|\leq (m+1) l_{n+1}$, 
\item  $B_{n+1}$ occurs in $x$ and $A_{n+1}$ occurs in $T^k y$ at the same place $i$ with $|i|\leq (m+1) l_{n+1}$.
\end{enumerate}

Without loss of generality we can assume that (1) holds. Recall that $A_{n+1} = A_n B_n A_n B_n$,  $B_{n+1} = B_n B_n B_n A_n$, and $l_{n+1} = 4 l_n -2$ (see figure 3.)

\setlength{\unitlength}{.4in}
\begin{picture}(5,1)(0,0)
\label{fig 1}
\linethickness{1pt}
\put(0,0){\line(1,0){15}}
\put(-.5, 0){\makebox(0,0){$x:$}}
\put(2,-.3){\makebox(0,0){$x_0$}}
\put(7,-.1){\line(0,1){0.2}}
\put(7,-.3){\makebox(0,0){$x_{i}$}}
\put(9.1,-.5){\makebox(0,0){$A_{n+1}$}}
\put(11.2,-.1){\line(0,1){0.2}}
\put(8.1,-.1){\line(0,1){0.2}}
\put(9.1,-.1){\line(0,1){0.2}}
\put(10.2,-.1){\line(0,1){0.2}}
\put(7.5,.3){\makebox(0,0){$A_{n}$}}
\put(8.5,.3){\makebox(0,0){$B_{n}$}}
\put(9.6,.3){\makebox(0,0){$A_{n}$}}
\put(10.6,.3){\makebox(0,0){$B_{n}$}}

\put(0,-2){\line(1,0){15}}
\put(-.7, -2){\makebox(0,0){$T^ky:$}}
\put(2,-2.4){\makebox(0,0){$(T^ky)_0$}}
\put(9,-2.4){\makebox(0,0){$B_{n+1}$}}
\put(11.1,-2.1){\line(0,1){0.2}}
\put(7,-2.1){\line(0,1){0.2}}
\put(8.0,-2.1){\line(0,1){0.2}}
\put(9.0,-2.1){\line(0,1){0.2}}
\put(10.0,-2.1){\line(0,1){0.2}}
\put(7.5,-1.7){\makebox(0,0){$B_{n}$}}
\put(8.5,-1.7){\makebox(0,0){$B_{n}$}}
\put(9.6,-1.7){\makebox(0,0){$B_{n}$}}
\put(10.6,-1.7){\makebox(0,0){$A_{n}$}}
\put(7,-2.4){\makebox(0,0){$(T^ky)_{i}$}}

\put(7, -3.0){\makebox(0,0){Figure 3}} 
\end{picture}

\vspace{3.5cm}

Let us consider the following 5 cases for $|k| \leq \frac{1}{2}l_{n+1} +1$:

\begin{enumerate}[{(a)}]
\item $-\frac{1}{8} l_{n+1} \leq k < \frac{1}{8} l_{n+1}$ ($\Rightarrow |k| \leq \frac{1}{8} l_{n+1}  \leq \frac{1}{2}l_n $): Then $A_nB_nA_n$ occurs at  $m_1 = i$ in $x$ and $B_nB_nB_n$ occurs at $m_2 = i -k$ in $y$ (see figure 4). We have $|m_1| \leq (|s_{n+1}| +1)l_{n+1} \leq (m+1)l_{n+1}$, $|m_2| \leq (|s_{n+1}| +1)l_{n+1} +|k| \leq (m+1) l_{n+1} + \frac{1}{8} l_{n+1} \leq (m+2)l_{n+1}$, and $|m_1 - m_2| = |k| \leq \frac{1}{2}l_n$, so $(i)$ holds.

\setlength{\unitlength}{.4in}
\begin{picture}(5,1)(0,0)
\label{fig 1}
\linethickness{1pt}
\put(0,0){\line(1,0){14}}
\put(-.5, 0){\makebox(0,0){$x:$}}
\put(2,-.3){\makebox(0,0){$x_0$}}
\put(7,-.1){\line(0,1){0.2}}
\put(7,-.3){\makebox(0,0){$x_{m_1}$}}
\put(11.2,-.1){\line(0,1){0.2}}
\put(8.1,-.1){\line(0,1){0.2}}
\put(9.1,-.1){\line(0,1){0.2}}
\put(10.2,-.1){\line(0,1){0.2}}
\put(7.6,.3){\makebox(0,0){$A_{n}$}}
\put(8.5,.3){\makebox(0,0){$B_{n}$}}
\put(9.6,.3){\makebox(0,0){$A_{n}$}}
\put(10.6,.3){\makebox(0,0){$B_{n}$}}

\put(0,-1.5){\line(1,0){14}}
\put(-.5, -1.5){\makebox(0,0){$y:$}}
\put(2,-1.2){\makebox(0,0){$y_0$}}
\put(6.7,-1.6){\line(0,1){0.2}}
\put(7.7,-1.6){\line(0,1){0.2}}
\put(8.7,-1.6){\line(0,1){0.2}}
\put(9.7,-1.6){\line(0,1){0.2}}
\put(10.8,-1.6){\line(0,1){0.2}}
\put(7.3,-1.8){\makebox(0,0){$B_{n}$}}
\put(8.2,-1.8){\makebox(0,0){$B_{n}$}}
\put(9.2,-1.8){\makebox(0,0){$B_{n}$}}
\put(10.3,-1.8){\makebox(0,0){$A_{n}$}}
\put(6.7,-1.2){\makebox(0,0){$y_{m_2}$}}

\put(6.7, -2.7){\makebox(0,0){Figure 4}} 
\end{picture}
\vspace{3.5cm}

\item $-\frac{3}{8} l_{n+1} \leq k < -\frac{1}{8} l_{n+1}$ ($\Rightarrow -\frac{3}{2}l_n + \frac{3}{4} \leq k < -\frac{1}{2}l_n + \frac{1}{4}$): Then $B_nA_nB_n$ occurs at $m_1=i+l_n$ in $x$ and $B_nB_nB_n$ occurs at $m_2=i-k$ in $y$ (see figure 5). We have $|m_1| \leq (m+1)l_{n+1} + l_{n}  \leq (m+2)l_{n+1}$ and $|m_2| \leq (m+1) l_{n+1}  + \frac{3}{8} l_{n+1} \leq (m+2)l_{n+1}$. Also, $m_1 - m_2 = l_n + k$ and $-\frac{1}{2}l_n + \frac{3}{4} \leq l_n + k < \frac{1}{2}l_n + \frac{1}{4}$, so $|m_1 - m_2| \leq \frac{1}{2}l_n+1$, so $(ii)$ holds.

\setlength{\unitlength}{.4in}
\begin{picture}(5,1)(0,0)
\label{fig 1}
\linethickness{1pt}
\put(0,0){\line(1,0){14}}
\put(-.5, 0){\makebox(0,0){$x:$}}
\put(2,-.3){\makebox(0,0){$x_0$}}
\put(7,-.1){\line(0,1){0.2}}
\put(8.1,-.3){\makebox(0,0){$x_{m_1}$}}
\put(8.1,-.1){\line(0,1){0.2}}
\put(9.1,-.1){\line(0,1){0.2}}
\put(10.2,-.1){\line(0,1){0.2}}
\put(11.2,-.1){\line(0,1){0.2}}
\put(7.6,.3){\makebox(0,0){$A_{n}$}}
\put(8.5,.3){\makebox(0,0){$B_{n}$}}
\put(9.6,.3){\makebox(0,0){$A_{n}$}}
\put(10.6,.3){\makebox(0,0){$B_{n}$}}

\put(0,-1.5){\line(1,0){14}}
\put(-.5, -1.5){\makebox(0,0){$y:$}}
\put(2,-1.2){\makebox(0,0){$y_0$}}
\put(8.3,-1.6){\line(0,1){0.2}}
\put(9.3,-1.6){\line(0,1){0.2}}
\put(10.3,-1.6){\line(0,1){0.2}}
\put(11.3,-1.6){\line(0,1){0.2}}
\put(12.4,-1.6){\line(0,1){0.2}}
\put(8.8,-1.8){\makebox(0,0){$B_{n}$}}
\put(9.8,-1.8){\makebox(0,0){$B_{n}$}}
\put(10.8,-1.8){\makebox(0,0){$B_{n}$}}
\put(11.9,-1.8){\makebox(0,0){$A_{n}$}}
\put(8.3,-1.2){\makebox(0,0){$y_{m_2}$}}

\put(6.7, -2.7){\makebox(0,0){Figure 5}} 
\end{picture}
\vspace{3.5cm}

\item $-\frac{1}{2} l_{n+1} -1 \leq k < -\frac{3}{8} l_{n+1}$ ($\Rightarrow -2 l_n  \leq k < -\frac{3}{2}l_n + \frac{3}{4}$): Then $B_{n+1} = B_n B_n B_n A_n$ occurs at $i -k$ in $y$ (see figure 6 (i)). 

\setlength{\unitlength}{.4in}
\begin{picture}(5,1)(0,0)
\label{fig 1}
\linethickness{1pt}
\put(0,0){\line(1,0){14}}
\put(-.5, 0){\makebox(0,0){$x:$}}
\put(2,-.3){\makebox(0,0){$x_0$}}
\put(7,-.1){\line(0,1){0.2}}
\put(7,-.3){\makebox(0,0){$x_{i}$}}
\put(8.1,-.1){\line(0,1){0.2}}
\put(9.1,-.1){\line(0,1){0.2}}
\put(10.2,-.1){\line(0,1){0.2}}
\put(11.2,-.1){\line(0,1){0.2}}
\put(7.6,.3){\makebox(0,0){$A_{n}$}}
\put(8.5,.3){\makebox(0,0){$B_{n}$}}
\put(9.6,.3){\makebox(0,0){$A_{n}$}}
\put(10.6,.3){\makebox(0,0){$B_{n}$}}

\put(0,-1.5){\line(1,0){14}}
\put(-.5, -1.5){\makebox(0,0){$y:$}}
\put(2,-1.2){\makebox(0,0){$y_0$}}
\put(8.9,-1.6){\line(0,1){0.2}}
\put(9.9,-1.6){\line(0,1){0.2}}
\put(10.9,-1.6){\line(0,1){0.2}}
\put(11.9,-1.6){\line(0,1){0.2}}
\put(13.0,-1.6){\line(0,1){0.2}}
\put(9.4,-1.8){\makebox(0,0){$B_{n}$}}
\put(10.4,-1.8){\makebox(0,0){$B_{n}$}}
\put(11.4,-1.8){\makebox(0,0){$B_{n}$}}
\put(12.5,-1.8){\makebox(0,0){$A_{n}$}}
\put(8.9,-1.2){\makebox(0,0){$y_{i-k}$}}

\put(6.7, -2.7){\makebox(0,0){Figure 6 (i)}} 
\end{picture}
\vspace{3.5cm}

Now consider the blocks occurring before the block $B_{n+1} = B_n B_n B_n A_n$ at $i-k$ in $y$. There are two possibilities: either $A_{n+1}=A_nB_nA_nB_n$ occurs right before $B_{n+1}$ or $B_{n+1}=B_nB_nB_nA_n$ occurs right before $B_{n+1}$ in $y$.

For the first case, $B_n$ occurs at $m_2 = i - k - l_n +1$ in $y$. Note that $B_nA_nB_n$ occurs at $m_1 = i+l_n$ in $x$ and $B_nB_nB_n$ occurs at $m_2 = i- k - l_n+1$ in $y$ (see figure 6 (ii)). Then $|m_1| \leq (m+1)l_{n+1} + l_n\leq (m+2)l_{n+1}$ and $|m_2| \leq (m+1)l_{n+1} + (\frac{1}{2}l_{n+1} +1) + l_n +1 \leq (m+2)l_{n+1}$. Also $m_1 -m_2 = k + 2l_n -1$. Then $-1 \leq k + 2 l_n -1 < \frac{1}{2}l_n - \frac{1}{4}$, so $|m_1 - m_2| \leq \frac{1}{2}l_n+2$, so $(iii)$ holds.

\setlength{\unitlength}{.4in}
\begin{picture}(5,1)(0,0)
\label{fig 1}
\linethickness{1pt}
\put(0,0){\line(1,0){14}}
\put(-.5, 0){\makebox(0,0){$x:$}}
\put(2,-.3){\makebox(0,0){$x_0$}}
\put(7,-.3){\makebox(0,0){$x_{i}$}}
\put(7,-.1){\line(0,1){0.2}}
\put(8.1,-.3){\makebox(0,0){$x_{m_1}$}}
\put(8.1,-.1){\line(0,1){0.2}}
\put(9.1,-.1){\line(0,1){0.2}}
\put(10.2,-.1){\line(0,1){0.2}}
\put(11.2,-.1){\line(0,1){0.2}}
\put(7.6,.3){\makebox(0,0){$A_{n}$}}
\put(8.5,.3){\makebox(0,0){$B_{n}$}}
\put(9.6,.3){\makebox(0,0){$A_{n}$}}
\put(10.6,.3){\makebox(0,0){$B_{n}$}}

\put(0,-1.5){\line(1,0){14}}
\put(-.5, -1.5){\makebox(0,0){$y:$}}
\put(2,-1.2){\makebox(0,0){$y_0$}}
\put(7.9,-1.6){\line(0,1){0.2}}
\put(8.9,-1.6){\line(0,1){0.2}}
\put(9.9,-1.6){\line(0,1){0.2}}
\put(10.9,-1.6){\line(0,1){0.2}}
\put(11.9,-1.6){\line(0,1){0.2}}
\put(13.0,-1.6){\line(0,1){0.2}}
\put(8.4,-1.8){\makebox(0,0){$B_{n}$}}
\put(9.4,-1.8){\makebox(0,0){$B_{n}$}}
\put(10.4,-1.8){\makebox(0,0){$B_{n}$}}
\put(11.4,-1.8){\makebox(0,0){$B_{n}$}}
\put(12.5,-1.8){\makebox(0,0){$A_{n}$}}
\put(7.9,-1.2){\makebox(0,0){$y_{m_2}$}}
\put(8.9,-1.2){\makebox(0,0){$y_{i-k}$}}

\put(6.7, -2.7){\makebox(0,0){Figure 6 (ii)}} 
\end{picture}
\vspace{3.5cm}

 Otherwise $B_nB_nB_nA_n$ occurs at $i - k - 4l_n -1$ of $y$. Let us consider the blocks occurring before $A_{n+1}=A_nB_nA_nB_n$ at $i$ in $x$; either $B_{n+1} = B_n B_n B_n A_n$ or $A_{n+1}=A_n B_n A_n B_n$ occurs before $A_{n+1}$ at $i$ in $x$. 
 
 For the first case $A_n$ occurs at $i-l_n$ of $x$. Then $A_nA_n$ occurs at $m_1=i-l_n$ in $x$ and $B_nB_n$ occurs at $m_2=i-k- 3l_n + 2$ of $y$ (see figure 6 (iii)). We have $|m_1| \leq (m+1)l_{n+1} + l_n \leq (m+2)l_{n+1}$ and $|m_2| \leq (m+1)l_{n+1} + (\frac{1}{2}l_{n+1} +1) + 3l_n +2 \leq (m+3)l_{n+1}$. Also $m_1-m_2 = k+ 2 l_n -2$, so $|m_1 - m_2| \leq \frac{1}{2}l_n +2$, so $(iii)$ holds. 
 
 \setlength{\unitlength}{.4in}
\begin{picture}(5,1)(0,0)
\label{fig 1}
\linethickness{1pt}
\put(0,0){\line(1,0){14}}
\put(-.5, 0){\makebox(0,0){$x:$}}
\put(2,-.3){\makebox(0,0){$x_0$}}
\put(5.9,-.3){\makebox(0,0){$x_{m_1}$}}
\put(6,-.1){\line(0,1){0.2}}
\put(7.1,-.1){\line(0,1){0.2}}
\put(7.1,-.3){\makebox(0,0){$x_{i}$}}
\put(8.1,-.1){\line(0,1){0.2}}
\put(9.1,-.1){\line(0,1){0.2}}
\put(10.2,-.1){\line(0,1){0.2}}
\put(11.2,-.1){\line(0,1){0.2}}
\put(6.6,.3){\makebox(0,0){$A_n$}}
\put(7.6,.3){\makebox(0,0){$A_{n}$}}
\put(8.5,.3){\makebox(0,0){$B_{n}$}}
\put(9.6,.3){\makebox(0,0){$A_{n}$}}
\put(10.6,.3){\makebox(0,0){$B_{n}$}}

\put(0,-1.5){\line(1,0){14}}
\put(-.5, -1.5){\makebox(0,0){$y:$}}
\put(2,-1.2){\makebox(0,0){$y_0$}}

\put(4.8,-1.6){\line(0,1){0.2}}
\put(5.8,-1.6){\line(0,1){0.2}}
\put(6.8,-1.6){\line(0,1){0.2}}
\put(7.8,-1.6){\line(0,1){0.2}}
\put(5.4,-1.8){\makebox(0,0){$B_{n}$}}
\put(6.4,-1.8){\makebox(0,0){$B_{n}$}}
\put(7.4,-1.8){\makebox(0,0){$B_{n}$}}
\put(8.5,-1.8){\makebox(0,0){$A_{n}$}}

\put(8.9,-1.6){\line(0,1){0.2}}
\put(9.9,-1.6){\line(0,1){0.2}}
\put(10.9,-1.6){\line(0,1){0.2}}
\put(11.9,-1.6){\line(0,1){0.2}}
\put(13.0,-1.6){\line(0,1){0.2}}
\put(9.4,-1.8){\makebox(0,0){$B_{n}$}}
\put(10.4,-1.8){\makebox(0,0){$B_{n}$}}
\put(11.4,-1.8){\makebox(0,0){$B_{n}$}}
\put(12.5,-1.8){\makebox(0,0){$A_{n}$}}
\put(8.9,-1.2){\makebox(0,0){$y_{i-k}$}}
\put(5.8,-1.2){\makebox(0,0){$y_{m_2}$}}

\put(6.7, -2.7){\makebox(0,0){Figure 6 (iii)}} 
\end{picture}
\vspace{3.5cm}

  For the latter case  $A_nB_nA_n$ occurs at $m_1 = i-2l_n+1$ in $x$ and $B_nB_nB_n$ occurs at $m_2 = i-k-4l_n+3$ in $y$ (see figure 6 (iv)). We have $|m_1| \leq (m+1)l_{n+1} + 2l_n +1 \leq (m+2)l_{n+1} $ and $|m_2| \leq (m+1)l_{n+1} +\frac{1}{2}l_{n+1}+1 + 4l_n+3\leq (m+3)l_{n+1}$. Also $m_1-m_2 = k+ 2 l_n -2$, so $|m_1 - m_2| \leq \frac{1}{2}l_n +2$, so $(i)$ holds. 

 \setlength{\unitlength}{.4in}
\begin{picture}(5,1)(0,0)
\label{fig 1}
\linethickness{1pt}
\put(0,0){\line(1,0){14}}
\put(-.5, 0){\makebox(0,0){$x:$}}
\put(2,-.3){\makebox(0,0){$x_0$}}
\put(4.9,-.1){\line(0,1){0.2}}
\put(6,-.1){\line(0,1){0.2}}
\put(4.9,-.3){\makebox(0,0){$x_{m_1}$}}
\put(7,-.1){\line(0,1){0.2}}
\put(7.0,-.3){\makebox(0,0){$x_{i}$}}
\put(8.1,-.1){\line(0,1){0.2}}
\put(9.1,-.1){\line(0,1){0.2}}
\put(10.2,-.1){\line(0,1){0.2}}
\put(11.2,-.1){\line(0,1){0.2}}
\put(5.6,.3){\makebox(0,0){$A_n$}}
\put(6.6,.3){\makebox(0,0){$B_n$}}
\put(7.6,.3){\makebox(0,0){$A_{n}$}}
\put(8.5,.3){\makebox(0,0){$B_{n}$}}
\put(9.6,.3){\makebox(0,0){$A_{n}$}}
\put(10.6,.3){\makebox(0,0){$B_{n}$}}

\put(0,-1.5){\line(1,0){14}}
\put(-.5, -1.5){\makebox(0,0){$y:$}}
\put(2,-1.2){\makebox(0,0){$y_0$}}

\put(4.8,-1.6){\line(0,1){0.2}}
\put(5.8,-1.6){\line(0,1){0.2}}
\put(6.8,-1.6){\line(0,1){0.2}}
\put(7.8,-1.6){\line(0,1){0.2}}
\put(5.4,-1.8){\makebox(0,0){$B_{n}$}}
\put(6.4,-1.8){\makebox(0,0){$B_{n}$}}
\put(7.4,-1.8){\makebox(0,0){$B_{n}$}}
\put(8.5,-1.8){\makebox(0,0){$A_{n}$}}

\put(8.9,-1.6){\line(0,1){0.2}}
\put(9.9,-1.6){\line(0,1){0.2}}
\put(10.9,-1.6){\line(0,1){0.2}}
\put(11.9,-1.6){\line(0,1){0.2}}
\put(13.0,-1.6){\line(0,1){0.2}}
\put(9.4,-1.8){\makebox(0,0){$B_{n}$}}
\put(10.4,-1.8){\makebox(0,0){$B_{n}$}}
\put(11.4,-1.8){\makebox(0,0){$B_{n}$}}
\put(12.5,-1.8){\makebox(0,0){$A_{n}$}}
\put(8.9,-1.2){\makebox(0,0){$y_{i-k}$}}
\put(4.8,-1.2){\makebox(0,0){$y_{m_2}$}}

\put(6.7, -2.7){\makebox(0,0){Figure 6 (iv)}} 
\end{picture}
\vspace{3.5cm}

\item  $\frac{1}{8} l_{n+1} \leq k < \frac{3}{8} l_{n+1}$ ( $\Rightarrow \frac{1}{2}l_n - \frac{1}{4} \leq k < \frac{3}{2}l_n - \frac{3}{4}$): 
$A_{n+1} = A_n B_n A_n B_n$ occurs at $i$ in $x$. We have $A_n$ or $B_n$ right before this block $A_{n+1}$ in $x$. For the first case, $A_nA_n$ occurs at $m_1= i- l_n $ in $x$ and $B_nB_n$ occurs at $m_2= i-k$ in $y$ (see the figure 7 (i)). Then $|m_i| \leq (m+2)l_{n+1}$ for $i=1,2$ and $m_2-m_1 = l_n-k$. Since $-\frac{1}{2}l_n + \frac{3}{4} <  l_{n}-k \leq \frac{1}{2}l_n + \frac{1}{4}$, we have $|m_1 -m_2| \leq \frac{1}{2}l_n + 2$, so $(iii)$ holds. 

\setlength{\unitlength}{.4in}
\begin{picture}(5,1)(0,0)
\label{fig 1}
\linethickness{1pt}
\put(0,0){\line(1,0){14}}
\put(-.5, 0){\makebox(0,0){$x:$}}
\put(2,-.3){\makebox(0,0){$x_0$}}
\put(5.9,-.1){\line(0,1){0.2}}
\put(5.9,-.3){\makebox(0,0){$x_{m_1}$}}
\put(7,-.3){\makebox(0,0){$x_{i}$}}
\put(7,-.1){\line(0,1){0.2}}
\put(8.1,-.1){\line(0,1){0.2}}
\put(9.1,-.1){\line(0,1){0.2}}
\put(10.2,-.1){\line(0,1){0.2}}
\put(11.2,-.1){\line(0,1){0.2}}
\put(6.5,.3){\makebox(0,0){$A_{n}$}}
\put(7.6,.3){\makebox(0,0){$A_{n}$}}
\put(8.5,.3){\makebox(0,0){$B_{n}$}}
\put(9.6,.3){\makebox(0,0){$A_{n}$}}
\put(10.6,.3){\makebox(0,0){$B_{n}$}}

\put(0,-1.5){\line(1,0){14}}
\put(-.5, -1.5){\makebox(0,0){$y:$}}
\put(2,-1.2){\makebox(0,0){$y_0$}}
\put(6.3,-1.6){\line(0,1){0.2}}
\put(7.3,-1.6){\line(0,1){0.2}}
\put(8.3,-1.6){\line(0,1){0.2}}
\put(9.3,-1.6){\line(0,1){0.2}}
\put(10.4,-1.6){\line(0,1){0.2}}
\put(6.8,-1.8){\makebox(0,0){$B_{n}$}}
\put(7.8,-1.8){\makebox(0,0){$B_{n}$}}
\put(8.8,-1.8){\makebox(0,0){$B_{n}$}}
\put(9.8,-1.8){\makebox(0,0){$A_{n}$}}
\put(6.3,-1.2){\makebox(0,0){$y_{m_2}$}}

\put(6.7, -2.7){\makebox(0,0){Figure 7 (i)}} 
\end{picture}
\vspace{3.5cm}

For the latter case, $B_nA_nB_n$ occurs at $m_1 = i-l_n+1$ in $x$ and $B_nB_nB_n$ occurs at $m_2= i-k $ in $y$ (see figure 7 (ii)). Similarly, we obtain $|m_i| \leq (m+2)l_{n+1}$ for $i=1,2$ and $|m_1 -m_2| \leq \frac{1}{2}l_n + 2$, so $(ii)$ holds.

\setlength{\unitlength}{.4in}
\begin{picture}(5,1)(0,0)
\label{fig 1}
\linethickness{1pt}
\put(0,0){\line(1,0){14}}
\put(-.5, 0){\makebox(0,0){$x:$}}
\put(2,-.3){\makebox(0,0){$x_0$}}
\put(6.0,-.1){\line(0,1){0.2}}
\put(6.0,-.3){\makebox(0,0){$x_{m_1}$}}
\put(7,-.3){\makebox(0,0){$x_{i}$}}
\put(7,-.1){\line(0,1){0.2}}
\put(8.1,-.1){\line(0,1){0.2}}
\put(9.1,-.1){\line(0,1){0.2}}
\put(10.2,-.1){\line(0,1){0.2}}
\put(11.2,-.1){\line(0,1){0.2}}
\put(6.5,.3){\makebox(0,0){$B_{n}$}}
\put(7.6,.3){\makebox(0,0){$A_{n}$}}
\put(8.5,.3){\makebox(0,0){$B_{n}$}}
\put(9.6,.3){\makebox(0,0){$A_{n}$}}
\put(10.6,.3){\makebox(0,0){$B_{n}$}}

\put(0,-1.5){\line(1,0){14}}
\put(-.5, -1.5){\makebox(0,0){$y:$}}
\put(2,-1.2){\makebox(0,0){$y_0$}}
\put(6.3,-1.6){\line(0,1){0.2}}
\put(7.3,-1.6){\line(0,1){0.2}}
\put(8.3,-1.6){\line(0,1){0.2}}
\put(9.3,-1.6){\line(0,1){0.2}}
\put(10.4,-1.6){\line(0,1){0.2}}
\put(6.8,-1.8){\makebox(0,0){$B_{n}$}}
\put(7.8,-1.8){\makebox(0,0){$B_{n}$}}
\put(8.8,-1.8){\makebox(0,0){$B_{n}$}}
\put(9.8,-1.8){\makebox(0,0){$A_{n}$}}
\put(6.3,-1.2){\makebox(0,0){$y_{m_2}$}}

\put(6.7, -2.7){\makebox(0,0){Figure 7 (ii)}} 
\end{picture}
\vspace{3.5cm}

\item $\frac{3}{8} l_{n+1} \leq k \leq \frac{1}{2} l_{n+1}$ ($\Rightarrow \frac{3}{2}l_n - \frac{3}{4} \leq k \leq 2l_n - 1$): $A_{n+1} = A_n B_n A_n B_n$ occurs at $i$ in $x$. Then either $B_{n+1}$ or $A_{n+1}$ occurs before this $A_{n+1}$.
 For the first case, $A_nA_n$ occurs at $m_1=i-l_n$ in $x$ and $B_nB_n$ occurs at $m_2=i - k + l_n -1$ in $y$ (see figure 8 (i)). Then, we have $|m_i| \leq (m+2)l_{n+1} $ and $m_2-m_1= -k + 2l_n -1$, so $|m_1 -m_2| \leq \frac{1}{2}l_n + 2$, so $(iii)$ holds.

\setlength{\unitlength}{.4in}
\begin{picture}(5,1)(0,0)
\label{fig 1}
\linethickness{1pt}
\put(0,0){\line(1,0){14}}
\put(-.5, 0){\makebox(0,0){$x:$}}
\put(2,-.3){\makebox(0,0){$x_0$}}
\put(5.9,-.1){\line(0,1){0.2}}
\put(7,-.1){\line(0,1){0.2}}
\put(5.9,-.3){\makebox(0,0){$x_{m_1}$}}
\put(7,-.3){\makebox(0,0){$x_{i}$}}
\put(11.2,-.1){\line(0,1){0.2}}
\put(8.1,-.1){\line(0,1){0.2}}
\put(9.1,-.1){\line(0,1){0.2}}
\put(10.2,-.1){\line(0,1){0.2}}
\put(6.5,.3){\makebox(0,0){$A_{n}$}}
\put(7.6,.3){\makebox(0,0){$A_{n}$}}
\put(8.5,.3){\makebox(0,0){$B_{n}$}}
\put(9.6,.3){\makebox(0,0){$A_{n}$}}
\put(10.6,.3){\makebox(0,0){$B_{n}$}}

\put(0,-1.5){\line(1,0){14}}
\put(-.5, -1.5){\makebox(0,0){$y:$}}
\put(2,-1.2){\makebox(0,0){$y_0$}}
\put(5.3,-1.6){\line(0,1){0.2}}
\put(6.3,-1.6){\line(0,1){0.2}}
\put(7.3,-1.6){\line(0,1){0.2}}
\put(8.3,-1.6){\line(0,1){0.2}}
\put(9.4,-1.6){\line(0,1){0.2}}
\put(5.8,-1.8){\makebox(0,0){$B_{n}$}}
\put(6.8,-1.8){\makebox(0,0){$B_{n}$}}
\put(7.8,-1.8){\makebox(0,0){$B_{n}$}}
\put(8.8,-1.8){\makebox(0,0){$A_{n}$}}
\put(5.3,-1.2){\makebox(0,0){$y_{i-k}$}}
\put(6.3,-1.2){\makebox(0,0){$y_{m_2}$}}

\put(6.7, -2.7){\makebox(0,0){Figure 8 (i)}} 
\end{picture}
\vspace{3.5cm}

For the latter case $A_nB_nA_n$ occurs at $m_1= i- 2l_n +1$ in $x$ and $B_nB_nB_n$ occurs at $m_2=i-k$ in $y$ (see figure 8 (ii)). Then $|m_i| \leq (m+2)l_{n+1}$ for $i=1,2$ and $|m_1 -m_2| \leq \frac{1}{2}l_n+2$, so $(i)$ holds.

\setlength{\unitlength}{.4in}

\begin{picture}(5,1)(0,0)
\label{fig 1}
\linethickness{1pt}
\put(0,0){\line(1,0){14}}
\put(-.5, 0){\makebox(0,0){$x:$}}
\put(2,-.3){\makebox(0,0){$x_0$}}
\put(5,-.1){\line(0,1){0.2}}
\put(6,-.1){\line(0,1){0.2}}
\put(7,-.1){\line(0,1){0.2}}
\put(5,-.3){\makebox(0,0){$x_{m_1}$}}
\put(7,-.3){\makebox(0,0){$x_{i}$}}
\put(11.2,-.1){\line(0,1){0.2}}
\put(8.1,-.1){\line(0,1){0.2}}
\put(9.1,-.1){\line(0,1){0.2}}
\put(10.2,-.1){\line(0,1){0.2}}
\put(5.5,.3){\makebox(0,0){$A_{n}$}}
\put(6.5,.3){\makebox(0,0){$B_{n}$}}
\put(7.6,.3){\makebox(0,0){$A_{n}$}}
\put(8.5,.3){\makebox(0,0){$B_{n}$}}
\put(9.6,.3){\makebox(0,0){$A_{n}$}}
\put(10.6,.3){\makebox(0,0){$B_{n}$}}

\put(0,-1.5){\line(1,0){14}}
\put(-.5, -1.5){\makebox(0,0){$y:$}}
\put(2,-1.2){\makebox(0,0){$y_0$}}
\put(5.3,-1.6){\line(0,1){0.2}}
\put(6.3,-1.6){\line(0,1){0.2}}
\put(7.3,-1.6){\line(0,1){0.2}}
\put(8.3,-1.6){\line(0,1){0.2}}
\put(9.4,-1.6){\line(0,1){0.2}}
\put(5.8,-1.8){\makebox(0,0){$B_{n}$}}
\put(6.8,-1.8){\makebox(0,0){$B_{n}$}}
\put(7.8,-1.8){\makebox(0,0){$B_{n}$}}
\put(8.8,-1.8){\makebox(0,0){$A_{n}$}}
\put(5.3,-1.2){\makebox(0,0){$y_{m_2}$}}

\put(6.7, -2.7){\makebox(0,0){Figure 8 (ii)}} 
\end{picture}

\vspace{2cm}
\qedhere
\end{enumerate} 
\end{proof}

\begin{proof}[Proof of Theorem \ref{eta}]
(a)  The proof is similar to the proof of Theorem \ref{msj theorem}. It is enough to show that $(X_{\eta}, \mathcal{X}, \mu, T)$ has minimal self-joinings of order 2. Let $\overline{\mu}$ be a twofold ergodic joining of $X_{\eta}$ and let $(x,y)$ be a $\overline{\mu}$-generic point. If $x$ and $y$ are on the same orbit under $T$, then $\overline{\mu}$ is off-diagonal.

Otherwise there exists an increasing sequence $(n_k)$ such that at least one of the conclusions $(i) -(iii)$ of Lemma \ref{orbit2} occurs. 
Suppose that the conclusion of $(i)$ of Lemma \ref{orbit2} occurs for a sequence $(n_k)$. We assume that $A_{n_k}B_{n_k}A_{n_k}$ occurs at $x_{m_1}$  and $B_{n_k}B_{n_k}B_{n_k}$ occurs at $y_{m_2}$ (see figure 9).

\setlength{\unitlength}{.4in}

\begin{picture}(5,4)(0,-2)
\label{fig 1}
\linethickness{1pt}
\put(0,0){\line(1,0){14}}
\put(-.5, 0){\makebox(0,0){$x:$}}
\put(2,.3){\makebox(0,0){$x_0$}}
\put(5,-.1){\line(0,1){0.2}}
\put(7.2,-.1){\line(0,1){0.2}}
\put(9.2,-.1){\line(0,1){0.2}}
\put(5,.3){\makebox(0,0){$x_{m_1}$}}
\put(11.4,-.1){\line(0,1){0.2}}
\put(6.0,.3){\makebox(0,0){$A_{n_k}$}}
\put(8.1,.3){\makebox(0,0){$B_{n_k}$}}
\put(10.2,.3){\makebox(0,0){$A_{n_k}$}}

\put(6.2,-0.6){\makebox(0,0){$M_k$}}
\put(5.5,-.2){\line(0,1){.2}}
\put(7.2,-.2){\line(0,1){.2}}
\put(5.5,-.2){\line(1,0){1.7}}
\put(10.5,-0.6){\makebox(0,0){$t_k +1+ M_k$}}
\put(9.7,-.2){\line(0,1){.2}}
\put(11.4,-.2){\line(0,1){.2}}
\put(9.7,-.2){\line(1,0){1.7}}

\put(0,-2){\line(1,0){14}}
\put(-.5, -2){\makebox(0,0){$y:$}}
\put(2,-2.3){\makebox(0,0){$y_0$}}
\put(5.5,-2.1){\line(0,1){0.2}}
\put(7.5,-2.1){\line(0,1){0.2}}

\put(9.5,-2.1){\line(0,1){0.2}}
\put(11.5,-2.1){\line(0,1){0.2}}
\put(6.5,-2.3){\makebox(0,0){$B_{n_k}$}}
\put(8.5,-2.3){\makebox(0,0){$B_{n_k}$}}
\put(10.5,-2.3){\makebox(0,0){$B_{n_k}$}}
\put(5.3,-2.3){\makebox(0,0){$y_{m_2}$}}

\put(6.2,-1.5){\makebox(0,0){$M_k$}}
\put(5.5,-2){\line(0,1){.2}}
\put(7.2,-2){\line(0,1){.2}}
\put(5.5,-1.8){\line(1,0){1.7}}
\put(10.3,-1.5){\makebox(0,0){$t_k + M_k$}}
\put(9.5,-2){\line(0,1){.2}}
\put(11.2,-2){\line(0,1){.2}}
\put(9.5,-1.8){\line(1,0){1.7}}

\put(6.7, -3.2){\makebox(0,0){Figure 9}} 
\end{picture}
\vspace{3.5cm}

Let $L_k=[-(m+4)l_{n_k+1}, (m+4)l_{n_k+1}]$ and let $M_k$ be the interval where the first $A_{n_k}$ of $A_{n_k}B_{n_k}A_{n_k}$ and the first $B_{n_k}$ of $B_{n_k}B_{n_k}B_{n_k}$ overlap. Choose $t_k = 2l_{n_k} - 2$. Applying the argument in Theorem \ref{msj theorem} with Lemma \ref{msj lemma}, we have that $\overline{\mu}$ is $T \times I$ invariant, so $\overline{\mu} = \mu \times \mu$. 
If one of the conclusions ($ii$) and ($iii$) of Lemma \ref{orbit2} holds for a sequence $(n_k)$, choose $t_k$
as ($ii$) $2l_{n_k} - 2$, or ($iii$) ${l_{n_k}}-1$ according the conclusion of Lemma \ref{orbit2} and apply Lemma \ref{msj lemma}, we have that $\overline{\mu}$ is $T \times I$ invariant or $I \times T$ invariant, so $\overline{\mu} = \mu \times \mu$.

(b) The proof of (b) follows along the lines of the proof of Theorem \ref{msj tiling}, combined with Lemma \ref{orbit2}, and is omitted.   
  \end{proof}

\section{rigidity and weak mixing}
In this section we construct an example of a tiling dynamical system which is rigid and weakly mixing.
Recall that $(X, \mathcal{X}, \mu, T_G)$ is said to be rigid if  there is a sequence $(t_i)$ in $G$ with $t_i \rightarrow \infty$ such that for every $f \in L^2(X)$, $f \circ T_{t_i} $ converges to $f$ in $L^2(X)$, which is equivalent to the fact that $\mu(A \cap T_{t_i}A) \rightarrow \mu(A)$ for any $A \in \mathcal{X}$. 

Let us consider first the following subshift, which was investigated in \cite{JR}.
A sequence of blocks $B_n$ is defined as follows:
$$B_1 = 010, B_{n+1} = B_n^{2^n} 1 B_n^{2^n}.$$
Let $X = \{ x=(x_n) \in \{0,1\}^{\mathbb{Z}} : \, \textrm{for all} \,\, i < j, \, x_i x_{i+1} \cdots x_j \,\, \textrm{occurs in some} \,\, B_k \}$.
Then $X$ is a closed shift invariant subset of $\{0,1\}^{\mathbb{Z}}$. It is known that $(X,T)$ is minimal and uniquely ergodic. Moreover, equipped with the shift invariant probability measure $\mu$ on the Borel $\sigma$-algebra $\mathcal{X}$, the measure preserving system $(X,\mathcal{X}, \mu,T)$ is rigid and weakly mixing (see \cite{JR}.)

Now let us consider a tiling $\mathcal{T}$ arising from a sequence $x=(x_n)$ in the space $X$ and two intervals $J_0$ and $J_1$. Note that the tiling dynamical system, which is denoted by $\bold{X} = (X_{\mathcal{T}}, \mathcal{D}, \nu, (T_t)_{(t \in \mathbb{R})})$, is minimal and uniquely ergodic since $(X,T)$ is minimal and uniquely ergodic.  

\begin{Theorem}
\label{rigid}
Let  $\bold{X} = (X_{\mathcal{T}}, \mathcal{D}, \nu, (T_t)_{(t \in \mathbb{R})})$ be a tiling dynamical system arising from the subshift $(X,T)$ and two intervals $J_0$ and $J_1$. If $ \frac{|J_0|}{|J_1|}$ is irrational, then the tiling dynamical system $\bold{X}$ is rigid and weakly mixing.
\end{Theorem}
 We will prove this for the case $|J_0| > |J_1|$. The proof for the other case ($|J_0| < |J_1|$) is analogous and will be omitted. In addition we assume that $J_0$ and $J_1$ have length $1$ and $\alpha$, where $\alpha$ is irrational and $0 < \alpha < 1$. We can do so because of the following fact: if $|\tilde{J_0}| = c |J_0|$ and $|\tilde{J_1}| = c |J_1|$ for some $c>0$, then $\lambda$ is an eigenvalue of  a tiling system with two intervals $J_0, J_1$ if and only if $\frac{\lambda}{c}$ is an eigenvalue of a tiling system with two intervals $\tilde{J_0}, \tilde{J_1}$.  
 
 Before proving Theorem \ref{rigid}, let us show the following lemma by an argument similar to that utilized in the proof of Lemma 1 in \cite{JP}.
\begin{Lemma}
\label{ergodic}
 $T_{\alpha}$ is an ergodic transformation on $(X_{{\mathcal{T}}},  {\mathcal{D}}, \nu)$.
 \end{Lemma}
 \begin{proof}
 Let $h_m = l(B_m)$. 
For each $m$, denote $B_m = p_1^m \cdots p_{h_m}^m$ and form a new alphabet $p_i^{m}$ ($1 \leq i \leq h_m$) as in the proof of Theorem \ref{msj theorem}. Then the countable set of patches 
\begin{equation*}
\begin{split}
\{[p_i^m] \times [0,r) : r \in \mathbb{Q} \,\, &\textrm {with} \,\, 0 < r < |J_a| \,\, \textrm {for} \,\, p_i^m = a \\
    &\textrm{or} \,\, 0 < r < |J_b| \,\, \textrm {for} \,\, p_i^m = b \,\, , 1 \leq i \leq l(A_m), m \in N \}
 \end{split}
\end{equation*}   
generates $\mathcal{D}$.

Consider a cylinder set $P = [p_i^n] \times [0,r)$. Pick a tiling $\mathcal{S} \in X_{{\mathcal{T}}}$.
We will show that $T_{\alpha}$ satisfies the ergodic theorem by showing that  $ T_{\alpha}^k (\mathcal{S})$ hits the cylinder set $P$ with the right frequency.

Let $R_n := I_1I_2 \cdots I_{h_n}$ be the patch corresponding to the block $B_n = x_1 x_2 \cdots x_{h_n}$ such that $I_i$ is equivalent to $J_{x_i}$. 
 For large $s$, consider a patch $R_s$. By Lemma 1.2 in \cite{JR}, for a fixed $n < s $, $R_s$ can be written uniquely as a concatenation of strings of the form $R_n^{2^n}$ or $R_n^{2^{n+1}}$, separated by single $J_1$s.  
 
  We estimate the number of $T_{\alpha}^k({\mathcal{S}})$ hitting the cylinder set $P$ by considering the translation $S_{\alpha}: x \rightarrow x+ \alpha$ on the patch $R_s$. There exists an interval $L \subset R_n$ and a point $y \in R_s$ such that $T_{\alpha}^k({\mathcal{S}})$ hitting the cylinder set $P$ along the patch $R_s$ is the same as $S_{\alpha}^k(x)$ hitting the interval $L$. 
  Then the sequence of numbers where $T_{\alpha}^k({\mathcal{S}})$ hits the cylinder set $P$ in $R_s$ converges as $s \rightarrow \infty$ since 
  \begin{enumerate}
  \item the irrational rotation $x \rightarrow x+ \alpha$ on the torus $[0, |R_n|)$ is uniquely ergodic,
  \item $S_{\alpha}^k(x)$ hits $J_{1}$ only once whenever it meets the interval $J_{1}$,
  \item $J_{1}$ appears in the patch $R_s$ with uniform frequency, which tends to $0$ as $s \rightarrow \infty$, 
  \item If $x'$ and $x''$ have the same relative positions in the first $R_n$ of the patches $R_nR_n$ and $R_nJ_1R_n$ respectively, then the relative positions of $S_{\alpha}^k(x')$ and $S_{\alpha}^{k+1}(x'')$ have the same relative positions when they occur in the second $R_n$ of the $R_nR_n$ and $R_nJ_1R_n$ respectively, since $|J_1| = \alpha$.  
  \end{enumerate}
 Hence the frequency $$\lim_{N \rightarrow \infty}\frac{1}{N} \sum_{n=1}^N 1_{P}(T_{\alpha}^{-k}(\mathcal{S}))= \lim_{N \rightarrow \infty}\frac{|\{ 1 \leq k \leq N : T_{\alpha}^k(\mathcal{S}) \in P\}|}{N}$$ exists.
 \end{proof}

\begin{proof}[Proof of Theorem \ref{rigid}] 

 Let $h_n = l(B_n)$ be the length of the block $B_n$. Note that
 \begin{equation}
 \label{length}
 h_{n+1} = 2^{n+1}h_n +1 \,\,\, \textrm{and} \,\,\, \lim\limits_{n \rightarrow \infty} h_n \mu([B_n]) = 1. 
 \end{equation}
The second equation is obtained from Lemma 1.2 and Corollary 1.7 in \cite{JR}:
By Lemma 1.2 therein, $h_n \mu([B_n]) = \bigcup\limits_{i=0}^{h_{n-1}} T^{i}([B_n])$ and by Corollary 1.7 of that paper  we have $$\lim_{n \rightarrow \infty} \mu(X \backslash \bigcup\limits_{i=0}^{h_{n-1}} T^{i}([B_n])) =0.$$

 Let $\mathcal{T}_n$ be the patch corresponding to the block $B_n$ and let $t_n$ be the diameter of $\mathcal{T}_n$. Let $\alpha_n$ and $\beta_n$ be the numbers of $0$s and $1$s in $B_n$. Note that $\alpha_1= 2, \beta_1= 1$, $\alpha_{n+1} = 2^{n+1} \alpha_n$ and $\beta_{n+1} = 1 + 2^{n+1} \beta_n$. Then $(\frac{\beta_n}{\alpha_n})$ is a bounded increasing sequence, since $\frac{\beta_{n+1}} {\alpha_{n+1}} = \frac{1}{2^{n+1}\alpha_{n}} + \frac{\beta_n}{\alpha_n} \leq \sum_{k=1}^n \frac{1}{2^k}=1$. Thus, 
\begin{equation}
\label{c}
\lim\limits_{n \rightarrow \infty} \frac{t_n}{h_n}
= \lim\limits_{n \rightarrow \infty} \frac{\alpha_n |J_0| + \beta_n |J_1|}{\alpha_n + \beta_n} 
=  \lim\limits_{n \rightarrow \infty} \frac{|J_0| +  (\beta_n / \alpha_n)|J_1|}{ 1+ (\beta_n / \alpha_n) } = c,
\end{equation}
for some $c>0$.

 Let us show first that $\bold{X}$ is rigid. 
 Let $C = [u] \times [0,r)$ for a finite word $u$ and  some $0 < r < |J_{u_0}|$. Since the measures of $[u]$ and $C$ are given by the occurence frequency, (\ref{tiling measure}) in Section 3.2 holds, so 
 $$\nu(C) = \nu ([u] \times [0,r)) = \mu([u]) \frac{r}{\mu([0])|J_0| + \mu([1])|J_1|}.$$
 
  Given $y=(y_n) \in X$, let $\mathcal{T}_y:= \{I_i : i \in \mathbb{Z}\}$ be the tiling of $\mathbb{R}$ corresponding the sequence $y$. 
  If $x \in [u] \cap T^{h_n} [u]$, then $T_t(\mathcal{T}_x) \in C \cap T_{t_n}C$ for $0 \leq t <r$, which implies that 
$$\nu(C \cap T_{t_n}C) \geq \mu([u] \cap T^{h_n}[u]) \frac{r}{\mu([0]) |J_0| + \mu([1]) |J_1|}.$$ 
 
 From Corollary 1.9 in \cite{JR}, $(h_n)$ is a rigidity sequence for $(X,T)$, so  $\lim\limits_{n \rightarrow \infty} 1_{[u]} \circ T^{h_n} = 1_{[u]}$ for any finite word $u$ in $X$. This implies that
$$\liminf \nu(C \cap T_{t_n} C) \geq \mu[u] \frac{r}{\mu([0]) |J_0| + \mu([1]) |J_1|} = \nu(C),$$ 
so $\bold{X}$ is rigid.

Now let us prove that $\bold{X} = (X_{\mathcal{T}}, \mathcal{D}, \nu, (T_t)_{(t \in \mathbb{R})})$ is weakly mixing. 
Let $f$ be an eigenfunction of $(T_t)_{t \in \mathbb{R}}$ with $T_t f = e^{2 \pi i \lambda t} f$ and $|f|=1$. Since $\{ B_n \times [a,b): n \in \mathbb{N}, a,b \in \mathbb{Q}^{+}\}$ generates the Borel $\sigma$-algebra $\mathcal{D}$, given $\epsilon >0$ we can find a positive integer $n$ and a bounded measurable function $f'$ such that $f'$ is a finite linear combination of indicator functions of cylinder sets and $\int |f - f'| \, d \nu \leq \epsilon$. Note that $f'$ is constant on $B_n \times \{ s \}$ for any $s$. 

Let $E_n = \{ \mathcal{S} : \mathcal{S} \in B_n^{2^n} 1 B_n^{2^n} \times [0, 2^n t_n) \}$ and $F_n =  \{ \mathcal{S} : \mathcal{S} \in B_n^{2^n} B_n^{2^n} \times [0, 2^n t_n) \}$. Note that $T_{2^n t_n + \alpha} f'(\mathcal{S}) = f'(\mathcal{S})$ for $\mathcal{S} \in E_n$ and $T_{2^n t_n} f'(\mathcal{S}) = f'(\mathcal{S})$ for $\mathcal{S} \in F_n$. 
Let us show that
 $\liminf \nu(E_n) \geq \frac{1}{2}d$ and $\liminf \nu(F_n) \geq \frac{1}{2}d$, where $d= \frac{c}{ \mu([0])||J_0|+\mu([1])|J_1|}$.
Since $B_{n+1}=B_n^{2^n} 1 B_n^{2^n}$, 
$$ \nu(E_n) \geq \mu(B_{n+1}) \frac{2^n t_n} { \mu([0])|J_0|+\mu([1])|J_1|}= \frac{1}{2} \mu(B_{n+1}) h_{n+1} \frac{2^{n+1} h_n}{h_{n+1}} \frac{t_n/h_n}{\mu([0])|J_0|+\mu([1])|J_1|}.$$
By (\ref{length}), we have $\lim\limits_{n \rightarrow \infty} \mu(B_{n+1}) h_{n+1} =1$ and $\lim\limits_{n \rightarrow \infty} \frac{2^{n+1}h_n}{h_{n+1}} = 1$. By (\ref{c}), we have $\lim\limits_{n \rightarrow \infty} \frac{t_n}{h_n} = c$. Thus,
$$\liminf\limits_{n \rightarrow \infty} \nu(E_n) \geq \frac{1}{2}d.$$
Note that  $B_{n+1}B_{n+1}= B_n^{2^n} 1 B_n^{2^n} B_n^{2^n} 1 B_n^{2^n}$ contains $B_n^{2^n}  B_n^{2^n}$ and the length of the patch corresponding to the block $B_n^{2^n}$ is $2^n t_n$, so $$\nu(F_n) \geq \mu(B_{n+1}B_{n+1}) \frac{2^n t^n}{\mu([0])|J_0|+\mu([1])|J_1|}.$$
Also $B_{n+1}B_{n+1}$ occurs $2(2^{n+1}-1)$-times in $B_{n+2}$, so $\mu(B_{n+1}B_{n+1}) \geq \mu(B_{n+2}) 2(2^{n+1}-1)$. Hence,
\begin{eqnarray*}
\nu(F_n) &\geq& \mu(B_{n+1}B_{n+1}) \frac{2^n t^n}{ \mu([0])|J_0|+\mu([1])|J_1|} \geq \mu(B_{n+2}) 2(2^{n+1}-1) \frac{2^n t_n}{ \mu([0])|J_0|+\mu([1])|J_1|} \\
&=& \frac{1}{2} \mu(B_{n+2}) h_{n+2} \frac{(2^{n+2}-2)2^{n+1}h_{n}}{h_{n+2}}  \frac{t_n/h_n}{ \mu([0])|J_0|+\mu([1])|J_1|}.
\end{eqnarray*}
By (\ref{length}), we have $\lim\limits_{n \rightarrow \infty} \mu(B_{n+2}) h_{n+2} =1$ and $\lim\limits_{n \rightarrow \infty}  \frac{(2^{n+2}-2)2^{n+1}h_{n}}{h_{n+2}} = \frac{(2^{n+2}-2)h_{n+1}}{h_{n+2}} \frac{2^{n+1}h_n}{h_{n+1}} =1$. By (\ref{c}), we have $\lim\limits_{n \rightarrow \infty} \frac{t_n}{h_n} = c$. Hence,
 $$\liminf\limits_{n \rightarrow \infty}\nu(F_n) \geq \frac{1}{2}d.$$
Then,
\begin{eqnarray*}
|1 - e^{2 \pi i \lambda (2^n t_n + \alpha)}| \, \nu(E_n) &=& \int_{E_n} |f - T_{2^n t_n+ \alpha} f | \, d \nu\\
&\leq& \int_{X_{\mathcal{T}}} |f - f'| \, d \nu + \int_{E_n} |f'- T_{2^n t_n+ \alpha}f' | \, d \nu +  \int_{X_{\mathcal{T}}} |T_{2^n t_n+ \alpha}f' - T_{2^n t_n+ \alpha}f | \, d \nu\\
&\leq& \int_{X_{\mathcal{T}}} |f - f'| \, d \nu + \int_{X_{\mathcal{T}}} |T_{2^n t_n+ \alpha}f' - T_{2^n t_n+ \alpha}f | \, d \nu   \leq 2 \epsilon,
\end{eqnarray*}
and 
\begin{eqnarray*}
|1 - e^{2 \pi i \lambda 2^n t_n}| \, \nu(F_n) &=& \int_{F_n} |f - T_{2^n t_n} f | \, d \nu\\
&\leq& \int_{X_{\mathcal{T}}} |f - f'| \, d \nu+ \int_{F_n} |f'- T_{2^n t_n}f' | \, d \nu +    \int_{X_{\mathcal{T}}} |T_{2^n t_n}f' - T_{2^n t_n}f| \, d \nu \\
&\leq& \int_{X_{\mathcal{T}}} |f - f'| \, d \nu + \int_{X_{\mathcal{T}}} |T_{2^n t_n}f' - T_{2^n t_n}f| \, d \nu \leq 2 \epsilon.
\end{eqnarray*}
Hence,
\begin{equation*}
|e^{2 \pi i  \lambda \alpha} - 1 | = |e^{2 \pi i \lambda (2^n t_n + \alpha)} - e^{2 \pi i \lambda 2^n t_n}| \leq |1 - e^{2 \pi i \lambda (2^n t_n + \alpha)}| + |1 - e^{2 \pi i \lambda 2^n t_n }| 
\leq \frac{2\epsilon}{\nu(E_n)} + \frac{2 \epsilon}{\nu(F_n)}. 
\end{equation*}
From  $\liminf\limits_{n \rightarrow \infty} \nu(E_n) \geq \frac{1}{2}d$ and $\liminf\limits_{n \rightarrow \infty} \nu(F_n) \geq \frac{1}{2}d$,
$$|e^{2 \pi i  \lambda \alpha} - 1 | \leq \frac{8}{d}\epsilon.$$
Since $\epsilon$ is arbitrarily small, $e^{2 \pi i  \lambda \alpha} = 1$. Then $T_{\alpha} f = f$. Ergodicity of $T_{\alpha}$ implies that $f$ is constant. Hence the dynamical system is weakly mixing.
\end{proof}

\renewcommand{\abstractname}{Acknowledgements}
\begin{abstract}
The author would like to thank Vitaly Bergelson for suggesting this research and for many useful discussions. The author also wishes to thank Donald Robertson for helpful comments on this paper.
\end{abstract}


\begin{thebibliography}{99}
\bibitem[BeRa]{BR} D. Berend and C. Radin, {\it{Are there chaotic tilings}}? Comm. Math. Phys. 152 (1993), no. 2, 215-219.

\bibitem[ClSa]{CS} A. Clark and L. Sadun, {\it{When size matters: subshifts and their related tiling spaces}}, Ergodic Theory Dynamical Systems 23 (2003), no. 4, 1043-1057.

\bibitem[DeKe]{DK} F. M. Dekking and M. Keane, {\it{Mixing properties of substitutions}}, Z. Wahrscheinlichkeitstheorie und Verw. Gebiete 42 (1978), no. 1, 23-33. 



\bibitem[Fu]{Fu} H. Furstenberg, Recurrence in ergodic theory and combinatorial number theory, Princeton University Press, Princeton, N.J., 1981.

\bibitem[Gl]{G} E. Glasner, Ergodic theory via joinings, Mathematical Surveys and Monographs, 101. American Mathematical Society, Providence, RI, 2003.

\bibitem[GlHoRu]{GHR} E. Glasner, B. Host, and D. Rudolph, {\it{Simple systems and their higher order self-joinings}}, Israel J. Math. 78 (1992), no. 1, 131-142. 



\bibitem[JaKe]{JaKe} K. Jacobs and M. Keane, {\it{0-1 Sequences of Toeplitz type}} Z. Wahrscheinlichkeitstheorie und Verw. Gebiete 13 (1969), 123-131.

\bibitem[JuPa]{JP} A. del Junco and K. Park, {\it{An example of a measure-preserving flow with minimal self-joinings}}, J. d'Analyse Math. 42 (1982/83), 199-209.

\bibitem[JuRaSw]{JRS} A. del Junco, M. Rahe, and L. Swanson, {\it{Chacon's automorphism has minimal self joinins}}, J. d'Analyse Math. 37 (1980), 276-284.

\bibitem[JuRo]{JR} A. del Junco and D. J. Rudolph, {\it{A rank one, rigid, simple, prime map}}, Ergodic Theory and Dynamical Systems 7 (1987), no. 2, 229-247.

\bibitem[Kak]{Kak} S. Kakutani, {\it{Strictly ergodic symbolic dynamical systems}}, Proc VI Berk. Sympos. Math. Statist. Probab. vol. II, 319-326, Univ. Calif. (1972).

\bibitem[KaSiSt]{KSS} A. B. Katok, Ya. G. Sinai and A. M. Stepin, {\it{Theory of dynamical systems and general transformation groups with invariant measure}}, J. Sov. Math. 7 (1977), 974-1065. 

\bibitem[Ki1]{Ki} J. King, {\it{The commutant is the weak closure of the powers, for rank one transformations}}, Ergodic Theory and Dynamical Systems 6 (1986), 363-384.

\bibitem[Ki2]{Ki2} J. King, {\it{Ergodic properties where order 4 implies infinite order}}, Israel J. Math. 80 (1992), no. 1-2, 65Ð86. 


\bibitem[Ox]{Ox} J. C. Oxtoby, {\it{Ergodic sets}}, Bull. Amer. Math. Soc. Volume 58, Number 2 (1952), 116-136.

\bibitem[Pe]{Pe} K. Petersen, Ergodic theory. Cambridge Studies in Advanced Mathematics, 2. Cambridge University Press, Cambridge, 1989


\bibitem[Que]{Q}M. Queff\'{e}lec, Substitution dynamical systems - spectral analysis. Second edition. Lecture Notes in Mathematics, 1294. Springer-Verlag, Berlin, 2010


\bibitem[Ro]{Ro2} E. A. Robinson, {\it{Symbolic dynamics and tilings of $\mathbb{R}^d$}}, Symbolic dynamics and its applications, 81Ð119, Proc. Sympos. Appl. Math., 60, Amer. Math. Soc., Providence, RI, 2004.

\bibitem[Ru1]{Ru1} D. J. Rudolph, {\it{An example of a measure-preserving map with minimal self-joinings, and applications}}, J. d'Analyse Math. 35 (1979), 97-122.

\bibitem[Ru2]{Ru} D. J. Rudolph, Fundamentals of measurable dynamics - Ergodic theory on Lebesque spaces, Oxford University Press, 1990. 

\bibitem[Ry]{Ry} V. V. Ryzhikov, {\it{Self-joinings of commutative actions with an invariant measure}}, Mat. Zametki 83 (2008), no.5, 792-795.


\bibitem[So]{S} B. Solomyak, {\it{Dynamics of self-similar tilings}}, Ergodic Theory and Dynamical Systems 17 (1997), 695-738.






\end{thebibliography}
\end{document}